\documentclass[12pt]{amsart}
\usepackage[utf8]{inputenc}
\usepackage{amsmath}
\usepackage{amsfonts}
\usepackage{amssymb}
\usepackage{amsthm}
\usepackage{bm}
\usepackage{bbm}
\usepackage{xcolor}
\usepackage{hyperref}
\usepackage{mathrsfs}
\theoremstyle{plain}
\newtheorem{theorem}{Theorem}
\newtheorem{lemma}[theorem]{Lemma}

\newtheorem{proposition}[theorem]{Proposition}
\theoremstyle{definition}

\theoremstyle{remark}

\author{Kyle Pratt}
\title{The irrationality of a divisor function series of Erd\H{o}s and Kac}
\address{All Souls College, Oxford OX1 4AL, UK}
\email{kyle.pratt@all-souls.ox.ac.uk}


\begin{document}
\date{}

\maketitle

\begin{abstract}
For positive integers $k$ and $n$ let $\sigma_k(n)$ denote the sum of the $k$th powers of the divisors of $n$. Erd\H{o}s and Kac asked whether, for every $k$, the number $\alpha_k = \sum_{n\geq 1} \frac{\sigma_k(n)}{n!}$ is irrational. It is known unconditionally that $\alpha_k$ is irrational if $k\leq 3$. We prove $\alpha_4$ is irrational.
\end{abstract}

\section{Introduction}

For positive integers $k$ and $n$ define $\sigma_k(n) = \sum_{d\mid n} d^k$, the sum of the $k$th powers of the divisors of $n$. Erd\H{o}s and Kac \cite{EK1954} conjectured that the number
\begin{align*}
\alpha_k := \sum_{n\geq 1} \frac{\sigma_k(n)}{n!}
\end{align*}
is irrational for every positive $k$. The irrationality of $\alpha_1$ and $\alpha_2$ is not so difficult to prove (see \cite{Erd1952,EK1954}), but it is more difficult to prove the irrationality of $\alpha_k$ for $k\geq 3$. Schlage-Puchta \cite{SP2006} and Friedlander, Luca, and Stoiciu \cite{FLS2007} independently proved the irrationality of $\alpha_3$ with arguments relying on sieve methods.

It was also proved in \cite{FLS2007, SP2006} that the irrationality of $\alpha_k$ follows in general from difficult conjectures in number theory, either Schinzel's Hypothesis $H$ \cite[Theorem]{SP2006} or an appropriate version of the Hardy-Littlewood prime $k$-tuples conjecture \cite[Theorem 2]{FLS2007}. Deajim and Siksek \cite[Theorem 2]{AS2011} gave a criterion (also conditional on Schinzel's Hypothesis $H$) for the set $\{1,\alpha_1,\ldots,\alpha_r\}$ to be linearly independent over $\mathbb{Q}$, and confirmed the criterion holds for $r=50$.

We prove unconditionally that $\alpha_4$ is irrational.
\begin{theorem}\label{thm:main theorem}
The number
\begin{align*}
\alpha_4 = \sum_{n\geq 1} \frac{\sigma_4(n)}{n!}=42.30104\ldots
\end{align*}
is irrational.
\end{theorem}

Our proof of Theorem \ref{thm:main theorem}, which is based on sieve methods, pushes the techniques to the limit and new ideas seem necessary to prove the irrationality of $\alpha_k$ for $k\geq 5$.

It is natural to place the irrationality of $\alpha_k$ in the context of $E$-functions, which are generalizations of the exponential function. Siegel \cite[p. 33]{Sieg1949} defined an $E$-function $f(z)$ to be a function of the form
\begin{align*}
f(z) = \sum_{n=0}^\infty \frac{a_n}{n!} z^n,
\end{align*}
where the $a_n$ are algebraic numbers in a fixed algebraic number field, for any fixed $\epsilon > 0$ we have $|a_n| = O_\epsilon(n^{n\epsilon})$, and there exists a sequence $q_n$ of positive integers with $q_n = O_\epsilon(n^{n\epsilon})$ such that $q_n a_n$ is an algebraic integer. Observe that $f(z)$ is an entire function.

Nowadays it is common to require an $E$-function to satisfy additionally a linear differential equation over $\overline{\mathbb{Q}}(z)$ (e.g. \cite{AR2018}). If an $E$-function $f(z)$ satisfies a first-order linear differential equation, then the powerful Siegel-Shidlovski theorem (see \cite[Chapter 3]{Shid1989}) implies $f(\alpha)$ is transcendental for any algebraic $\alpha$ which is not in some explicit finite set (depending on $f$). If $f$ is the solution of a linear differential equation of higher order then the transcendence results are not quite as strong, but are still in a relatively satisfactory state \cite{AR2018}.

In light of this, it seems natural to study the entire functions
\begin{align*}
f_k(z) = \sum_{n=0}^\infty \frac{\sigma_k(n)}{n!} z^n.
\end{align*}
The conjecture of Erd\H{o}s and Kac is then that $f_k(1)$ is irrational for every positive integer $k$. However, the functions $f_k(z)$ do not appear to satisfy any suitable differential equations which are susceptible to the Siegel-Shidlovskii technique. If we define the more general functions
\begin{align*}
f_{k,\ell}(z) = \sum_{n=0}^\infty \frac{\sigma_k(n) n^\ell}{n!} z^n
\end{align*}
then we do have the obvious differential equations
\begin{align*}
\frac{d}{dz}f_{k,\ell}(z) = z^{-1}(f_{k,\ell+1}(z)-1),
\end{align*}
but since the index $\ell$ increases without bound under repeated differentiation these functions are also unsuitable. Even so, one expects
\begin{align*}
f_k(\alpha) = \sum_{k=0}^\infty \frac{\sigma_k(n)}{n!} \alpha^n
\end{align*}
to be transcendental for any non-zero algebraic number $\alpha$ and positive integer $k$, but this is far out of reach of present techniques. We therefore interpret the Erd\H{o}s-Kac conjecture as pointing towards irrationality and transcendence results for $E$-functions which lack differential structure.

Siegel also introduced $G$-functions, which generalize the geometric series as $E$-functions generalize the exponential function (see e.g. \cite{ChCh1985} for more on $G$-functions). A $G$-function $f(z)$ is a power series of the form $\sum_{n=0}^\infty a_n z^n$ where the algebraic numbers $a_n$ have at most exponential growth, and the denominators of the $a_n$ also do not grow too quickly. As with $E$-functions, it is common to assume some differential structure as well. The theory of irrationality and transcendence for $G$-functions is much less satisfactory than that for $E$-functions. 

Many other irrationality conjectures of Erd\H{o}s fit comfortably within the context of $G$-functions. For instance, Erd\H{o}s conjectured \cite{Erd1986} that the numbers
\begin{align*}
\sum_{n\geq 1} \frac{\sigma_1(n)}{2^n}, \ \ \ \ \ \sum_{n\geq 1} \frac{\varphi(n)}{2^n}, \ \ \ \ \ \ \sum_{n\geq 1} \frac{\omega(n)}{2^n}
\end{align*}
are irrational, where $\varphi(n)$ is the Euler totient function and $\omega(n)$ is the number of distinct prime divisors of $n$. At present, only the first of these numbers is known to be irrational, which is a corollary of deep work of Nesterenko \cite[Theorem 1]{Nes1996} on the algebraic independence of values of Eisenstein series. Thus, the irrationality conjectures of Erd\H{o}s, while often appearing \emph{ad hoc} at first glance, hint at deeper mathematical problems.

In section \ref{sec:notation} we collect the notation and sieve results we use in the paper; this should be skipped on a first reading and then consulted as necessary. In section \ref{sec:outline of proof} we give the proof of Theorem \ref{thm:main theorem} assuming five propositions. One of the propositions is proved in section \ref{sec:outline of proof}, and the other propositions are proved in later sections. We give a brief outline of the remaining sections of the paper at the end of section \ref{sec:outline of proof}.

\section{Notation and background}\label{sec:notation}

In this section we describe the notation in force throughout the paper, and we state the basic sieve results that form the backbone of our work.

We write $\epsilon>0$ for a sufficiently small constant which is fixed throughout the paper; probably taking $\epsilon = 10^{-100}$ would suffice. The real number $x$ is always sufficiently large depending on every other fixed quantity such as $\epsilon$ or $\alpha_4$. We define
\begin{align}\label{eq:defn of D0}
D_0 = \frac{1}{2}\log \log x
\end{align}
and
\begin{align}\label{eq:defn of W}
W &= 12 \prod_{5\leq p\leq D_0} p.
\end{align}

We write $f \ll g, g \gg f$, or $f = O(g)$ if there is a constant $C>0$ such that $|f| \leq C |g|$. If $f = O(g)$ and $g = O(f)$ then we write $f \asymp g$. If the implied constant depends on some other quantity or parameter we often denote this via a subscript, such as $f \ll_\epsilon g$.

For $\theta \in \mathbb{R}$ we write $\| \theta \| = \min_{n \in \mathbb{Z}} |\theta - n|$. Given $x \in \mathbb{R}$ we define $e(x) = e^{2\pi i x}$. For a condition $C(n)$, we write $\mathbf{1}_{C(n)}$ for the indicator function of this event. The cardinality of a finite set $S$ is written as $|S|$.

The letters $p,q,r,t$ are reserved exclusively for prime numbers, while $\ell,m,n$ are integers.

We write $m \mid n$ if $m$ divides $n$ and $m \nmid n$ if $m$ does not divide $n$. We write $\varphi(n)$ for the Euler totient or phi function, $\tau(n)$ for the number of divisors of $n$, and $\mu(n)$ for the M\"obius function. The number of distinct prime factors of $n$ is $\omega(n)$. The least prime divisor of $n$ is $P^-(n)$. The greatest common divisor of two integers $m$ and $n$ is $(m,n)$. 

We often write congruences $n\equiv a \pmod{Q}$ as $n\equiv a (Q)$.

For a real $z\geq 2$ we write $P(z) = \prod_{p\leq z} p$ and for $2\leq w < z$ we write $P(w,z) = \prod_{w<p\leq z} p$.

A positive integer $n$ is $y$-smooth if every prime factor $p$ of $n$ satisfies $p\leq y$. The number of $y$-smooth integers $\leq x$ is denoted by $\Psi(x,y)$. We write $\rho(u)$ for the Dickman-de Bruijn function. The function $\rho(u)$ is differentiable and satisfies the differential delay equation
\begin{align}\label{eq:dickman rho differential delay}
u\rho'(u) + \rho(u-1) = 0
\end{align}
with the initial condition $\rho(u)=1$ for $0\leq u \leq 1$ (see \cite[(1.5)-(1.6)]{HT1993}).

Our arguments rely heavily on sieve methods. An important tool is the linear sieve, which we use in a sharp form due to Rosser-Iwaniec. The following convenient formulation is essentially \cite[Lemma 2.1]{LWC2019} (see also \cite[Theorem 11.12]{FI2010} and \cite[p. 235]{FI2010}).

\begin{lemma}[Linear sieve]\label{lem:linear sieve}
Let $D\geq 2$ and $L>1$. Let $\mathscr{P}$ denote a set of primes, let $z\geq 2$ and write $P(z) = \prod_{\substack{p\leq z \\ p \in \mathscr{P}}}p$. There exist two sequences of real numbers $\{\lambda_d^\pm\}_{d=1}^\infty$, vanishing for $d>D$ or $\mu(d)=0$, satisfying $\lambda_1^\pm = 1$, $|\lambda_d^\pm|\leq 1$,
\begin{align}\label{eq:linear sieve inequalities}
\sum_{d\mid n}\lambda_d^-\leq \sum_{d\mid n} \mu(d)\leq \sum_{d\mid n}\lambda_d^+,
\end{align}
and such that
\begin{align*}
\sum_{d\mid P(z)}\lambda_d^+ \frac{w(d)}{d}&\leq\Big\{F(s) + O_L\Big((\log D)^{-1/6} \Big)\Big\} \prod_{\substack{p\leq z \\ p \in \mathscr{P}}}\left(1 - \frac{w(p)}{p}\right), \\
\sum_{d\mid P(z)}\lambda_d^- \frac{w(d)}{d}&\geq\Big\{f(s) - O_L\Big((\log D)^{-1/6} \Big)\Big\} \prod_{\substack{p\leq z \\ p \in \mathscr{P}}}\left(1 - \frac{w(p)}{p}\right)
\end{align*}
uniformly for all multiplicative functions $w$ satisfying $0 < w(p) < p, p \in \mathscr{P}$, and
\begin{align*}
\prod_{\substack{u < p \leq v \\ p \in \mathscr{P}}}\Big(1 - \frac{w(p)}{p} \Big)^{-1} \leq \frac{\log v}{\log u} \left(1 + \frac{L}{\log u}\right), \ \ \ \ \ 2\leq u \leq v \leq z.
\end{align*}
Here $s = \frac{\log D}{\log z}$, and $F(s), f(s)$ are given by the continuous solutions to the system
\begin{align*}
\begin{cases}
sF(s) = 2e^\gamma, \ \ \ \ \ &1\leq s \leq 3, \\
sf(s) = 0, &0\leq s \leq 2, \\
(sF(s))' = f(s-1), &s>3, \\
(sf(s))' = F(s-1), &s>2.
\end{cases}
\end{align*}
\end{lemma}

We also use the fundamental lemma of the sieve to accurately sift small primes \cite[Section 6.5]{FI2010}. All our applications will be small variations on the inequalities
\begin{align}\label{eq:fundamental lemma}
\sum_{\substack{d\mid P((\log x)^{100}) \\ d \mid n \\ \omega(d) \leq 2R+1}}\mu(d)\leq \mathbf{1}_{(n,P((\log x)^{100}))=1} \leq \sum_{\substack{d\mid P((\log x)^{100}) \\ d \mid n \\ \omega(d) \leq 2R}}\mu(d),
\end{align}
where $R$ is any positive integer \cite[(6.6)]{FI2010}.

We refer the reader to the standard source \cite{FI2010} for additional information about sieves.

\section{Outline of the proof}\label{sec:outline of proof}

Our method has some similarities to the proofs of irrationality of $\alpha_3$ \cite{FLS2007,SP2006}, but is much more complicated. In particular, a sieve method lies at the heart of our approach, but the finishing blow is provided by an appeal to the theory of exponential sums.

In this section we prove Theorem \ref{thm:main theorem}, assuming five propositions. The following proposition is the foundation of all our later work.

\begin{proposition}[Rationality implies near-integrality]\label{prop:special p implies close to integer}
Assume that $\alpha_4$ is rational. Let $\epsilon > 0$ be a sufficiently small constant. Let $x$ be sufficiently large depending on $\alpha_4$ and $\epsilon$, and assume $p \in (x/2,x]$ is a prime such that 
\begin{itemize}
\item $p+2$ is squarefree, all the prime factors of $p+2$ are $> x^{1/4-\epsilon}$, and $p+2$ has at most one prime factor in the interval $(x^{1/4-\epsilon} , x^{1/4}(\log x)^{100}]$,
\item $\frac{p+3}{2}$ has no prime factors $\leq (\log x)^{100}$.
\end{itemize}
If $p+2$ has no prime factor $\leq x^{1/4}(\log x)^{100}$ then
\begin{align}\label{eq:p+1 frac part with no small prime factors}
\Big\|\frac{\sigma_4(p+1)}{p(p+1)} + \frac{1}{16} \Big\| \leq (\log x)^{-100}.
\end{align}
If $p+2$ has a prime factor $r$ in the interval $(x^{1/4-\epsilon} , x^{1/4}(\log x)^{100}]$, then
\begin{align}\label{eq:p+1 frac part with one small prime factor}
\Big\|\frac{\sigma_4(p+1)}{p(p+1)} +\frac{p+1}{r^4}+ \frac{1}{16} \Big\| \leq (\log x)^{-100}.
\end{align}

\end{proposition}
\begin{proof}
If $\alpha_4$ is rational then $\alpha_4 = a/b$ for some positive integers $a$ and $b$.  For a prime $p$ as in the statement of the proposition, consider
\begin{align*}
N=(p-1)!\sum_{n=p}^\infty \frac{\sigma_4(n)}{n!} = (p-1)! \Big(\alpha_4 - \sum_{n=1}^{p-1} \frac{\sigma_4(n)}{n!} \Big) = (p-1)!\frac{a}{b} - \sum_{n=1}^{p-1} \sigma_4(n)\frac{(p-1)!}{n!}.
\end{align*}
Since $p\asymp x$ is sufficiently large the number on the right-hand side is an integer, and therefore $N$ is a (manifestly positive) integer.

We rewrite $N$ as
\begin{align*}
N &= \frac{\sigma_4(p)}{p} + \frac{\sigma_4(p+1)}{p(p+1)} + \frac{\sigma_4(p+2)}{p(p+1)(p+2)} + \frac{\sigma_4(p+3)}{p(p+1)(p+2)(p+3)} \\ 
&+ \sum_{j=4}^\infty \frac{\sigma_4(p+j)}{p(p+1)\cdots(p+j)}.
\end{align*}
Since $\sigma_4(m) = \sum_{d\mid m} d^4 = m^4 \sum_{d\mid m} d^{-4} \leq m^4 \sum_{d\geq 1} d^{-4} \ll m^4$ we see that the sum over $j$ has size $\ll p^{-1} \ll x^{-1}$. Since $p$ is a prime
\begin{align*}
\frac{\sigma_4(p)}{p} = \frac{p^4 + 1}{p} = p^3 + p^{-1} = p^3 + O(x^{-1}),
\end{align*}
and it follows that
\begin{align*}
\Big\|\frac{\sigma_4(p+1)}{p(p+1)} + \frac{\sigma_4(p+2)}{p(p+1)(p+2)} + \frac{\sigma_4(p+3)}{p(p+1)(p+2)(p+3)} \Big\| \ll x^{-1}.
\end{align*}

The term involving $p+3$ is also straightforward. By multiplicativity and easy estimation, we have
\begin{align*}
\frac{\sigma_4(p+3)}{p(p+1)(p+2)(p+3)} &= (1+O(x^{-1}))\frac{\sigma_4(2)}{2^4} \frac{\sigma_4(\frac{p+3}{2})}{(\frac{p+3}{2})^4} = \frac{17}{16} + O((\log x)^{-200}),
\end{align*}
the second equality following from the fact that $\frac{p+3}{2}$ has no prime factors $\leq (\log x)^{100}$.

Next we turn to the term involving $p+2$. Note that
\begin{align*}
\frac{\sigma_4(p+2)}{p(p+1)(p+2)} &= \frac{\sigma_4(p+2)}{(p+2)^3} + 3\frac{\sigma_4(p+2)}{(p+2)^4} + O(p^{-1}) = \frac{\sigma_4(p+2)}{(p+2)^3} + 3 + O(x^{-1/4+\epsilon}),
\end{align*}
since $p+2$ has no prime divisors $\leq x^{1/4-\epsilon}$. If $p+2$ is squarefree and has no prime factors $\leq x^{1/4}(\log x)^{100}$, then $p+2$ has at most three distinct prime factors and we may write $p+2 = q_1 \cdots q_k$ with $q_1 > \cdots > q_k$ and $1\leq k \leq 3$. Therefore
\begin{align*}
\frac{\sigma_4(p+2)}{(p+2)^3} &= \prod_{i=1}^k \frac{q_i^4+1}{q_i^3} = \prod_{i=1}^k \left( q_i + \frac{1}{q_i^3}\right) = q_1 \cdots q_k + O\left(\frac{q_1\cdots q_{k-1}}{q_k^3} \right).
\end{align*}
The error term is
\begin{align*}
\frac{q_1\cdots q_{k-1}}{q_k^3} = \frac{q_1 \cdots q_k}{q_k^4} = \frac{p+2}{q_k^4} \ll (\log x)^{-400}
\end{align*}
since $q_k > x^{1/4}(\log x)^{100}$. It follows that
\begin{align*}
\Big\|\frac{\sigma_4(p+1)}{p(p+1)} + \frac{1}{16} \Big\| \leq (\log x)^{-100}
\end{align*}
if $p+2$ has no prime factor $\leq x^{1/4}(\log x)^{100}$.

Now assume there exists a prime divisor $r$ of $p+2$ in $(x^{1/4-\epsilon}, x^{1/4}(\log x)^{100}]$. Since $p+2$ is squarefree and has at most one prime divisor $\leq x^{1/4}(\log x)^{100}$, we may write $p+2 = q_1 \cdots q_k r$ where $q_1 > \cdots > q_k > r$ and $q_k > x^{1/4}(\log x)^{100}$. Then
\begin{align*}
\frac{\sigma_4(p+2)}{(p+2)^3} &= \left(r + \frac{1}{r^3}\right)\prod_{i=1}^k \left( q_i + \frac{1}{q_i^3}\right) = q_1 \cdots q_k r + \frac{q_1 \cdots q_k}{r^3} + O\left(\frac{q_1 \cdots q_{k-1}r}{q_k^3} \right) \\
&= p+2 + \frac{p+2}{r^4} + O((\log x)^{-400}).
\end{align*}
This yields the claim of the proposition when $p+2$ has a prime divisor $\leq x^{1/4}(\log x)^{100}$.
\end{proof}

The overall strategy of the proof of Theorem \ref{thm:main theorem}, then, is to assume for contradiction that $\alpha_4$ is rational and then show there exists a prime $p$ satisfying the hypotheses of Proposition \ref{prop:special p implies close to integer} such that neither \eqref{eq:p+1 frac part with no small prime factors} nor \eqref{eq:p+1 frac part with one small prime factor} is satisfied.

Ideally, we would like to impose in Proposition \ref{prop:special p implies close to integer} the condition that $p+2$ has no prime divisors $\leq x^{1/4+\epsilon}$, say, since then we would not need to bifurcate into cases. However, showing the existence of primes $p$ such that $p+2$ has no prime factors $\leq x^{1/4+\epsilon}$ brushes up against the sifting limits of the linear sieve; it is not currently possible to prove the existence of such primes without using strong results on the distribution of primes in arithmetic progressions. Some results are available \cite[Theorem 10]{BFI1986}, but these require a substantial amount of mathematical technology such as the well-factorable form of the linear sieve weights \cite{Iwa1980} and estimates for sums of Kloosterman sums coming from the spectral theory of automorphic forms \cite{DI1982}.

We prefer to give a proof which requires less firepower, and therefore we use a more standard result on primes in arithmetic progressions, namely the Bombieri-Vinogradov theorem \cite[Chapter 28]{Dav2000}. Using the Bombieri-Vinogradov theorem and some sieve theory arguments allows us to find many primes $p$ satisfying the hypotheses of Proposition \ref{prop:special p implies close to integer}.

For technical reasons it is convenient to count primes $p\asymp x$ in a residue class modulo a slowly growing integer $W$ (see \eqref{eq:defn of W}) which is divisible by all the primes up to $D_0$ (see \eqref{eq:defn of D0}). This is the so-called ``$W$-trick,'' which can be a useful technical device in some sieving contexts (see e.g. \cite{May2015}). We note that by the prime number theorem we have the bounds $(\log x)^{1/3} \leq W \leq (\log x)^{2/3}$. The $W$-trick ensures that $p,p+2$, and $p+3$ have no very small prime factors, other than the ones they are ``forced'' to have (for instance, $p+3$ is divisible by 2, and at least one of $n,n+1,n+2,n+3$ must be divisible by 3).

\begin{proposition}[Existence of many special primes]\label{prop:lower bound on number of good primes}
Let $\epsilon>0$ be sufficiently small and fixed, and let $x$ be sufficiently large depending on $\epsilon$. Let $\mathcal{S}$ denote the set of primes $p \in (x/2,x]$ satisfying
\begin{itemize}
\item $p\equiv -1 \pmod{W}$,
\item $p+2$ is squarefree, all the prime factors of $p+2$ are $> x^{1/4-\epsilon}$, and $p+2$ has at most one prime factor in the interval $(x^{1/4-\epsilon} , x^{1/4}(\log x)^{100}]$,
\item $\frac{p+3}{2}$ has no prime factors $\leq (\log x)^{100}$.
\end{itemize}
Then
\begin{align*}
|\mathcal{S}| &\geq 30 \epsilon \cdot e^{-\gamma}\frac{W^2}{\varphi(W)^3} \frac{x}{2(\log x)^2 \log ((\log x)^{100})}.
\end{align*}
\end{proposition}

An important consequence of Proposition \ref{prop:special p implies close to integer} is that if $\alpha_4$ is rational then $|\mathcal{S}| \leq \Sigma_1 + \Sigma_2$, where
\begin{align}\label{eq:defn of Sigma 1}
\Sigma_1 &:= \sum_{\substack{x/2 < p \leq x \\ p\equiv -1 (W) \\ \|\frac{\sigma_4(p+1)}{p(p+1)} + \frac{1}{16} \| \leq (\log x)^{-100}}} \mathbf{1}_{P^-(p+2) > x^{1/4}(\log x)^{100}}\cdot \mathbf{1}_{P^-(\frac{p+3}{2}) > (\log x)^{100}},
\end{align}
and
\begin{align*}
\Sigma_2 &:= \sum_{x^{1/4-\epsilon} < r \leq x^{1/4}(\log x)^{100}}\sum_{\substack{x/2 < p \leq x \\ p\equiv -1 (W) \\ p+2 \equiv 0 (r) \\ \|\frac{\sigma_4(p+1)}{p(p+1)} + \frac{p+1}{r^4} + \frac{1}{16} \| \leq (\log x)^{-100}}} \mathbf{1}_{P^-(\frac{p+2}{r}) > x^{1/4}(\log x)^{100}}\cdot \mathbf{1}_{P^-(\frac{p+3}{2}) > (\log x)^{100}},
\end{align*}
where the sum over $r$ in $\Sigma_2$ is over primes. Let us first sketch how we treat the simpler sum $\Sigma_1$. The key point is that, for almost all $p$, the number $p+1$ will have a prime factor $q$ in a convenient range (imagine $q \approx x^\epsilon$ for the purposes of this outline). If $p+1$ has a prime factor $q$ of convenient size, then we can factor $p+1 = \ell q$ and write
\begin{align*}
\frac{\sigma_4(p+1)}{p(p+1)} \approx \frac{\sigma_4(p+1)}{(p+1)^2} =\frac{\sigma_4(\ell)\sigma_4(q)}{\ell^2 q^2} = \frac{\sigma_4(\ell)}{\ell^2} (q^2 + q^{-2}),
\end{align*}
at least in the typical case where $\ell$ and $q$ are coprime. We can then control condition \eqref{eq:p+1 frac part with no small prime factors} by bounding exponential sums of the form
\begin{align*}
\sum_{q \asymp Q}e\left(A \left( q^2 + q^{-2}\right) \right),
\end{align*}
for some quantity $A$ which is much larger than $Q$ and which is independent of $q$. Since the phase function is a smooth function of $q$, these exponential sums are susceptible to the classical methods of Weyl-van der Corput. Making this argument precise yields the following proposition.

\begin{proposition}[$\Sigma_1$ is small]\label{prop:sigma 1 is small}
Let $\Sigma_1$ be defined as in \eqref{eq:defn of Sigma 1}. Let $\epsilon>0$ be sufficiently small and fixed, and let $x$ be sufficiently large depending on $\epsilon$. Then
\begin{align*}
\Sigma_1 &\ll \epsilon^2\frac{W^2}{\varphi(W)^3} \frac{x}{(\log x)^2 \log\log x},
\end{align*}
where the implied constant is absolute.
\end{proposition}

It remains to treat $\Sigma_2$. The argument here is more complicated, but the basic idea is still to introduce a conveniently-sized prime factor of $p+1$ and reduce to exponential sum estimates. In order to make room for the congruence condition to modulus $r \approx x^{1/4}$ and to still have room left over for inserting sieve weights, we want $p+1$ to have a prime factor greater than $x^{3/10}$, say. That is, we do not want $p+1$ to be $x^{3/10}$-smooth. We therefore split
\begin{align*}
\Sigma_2 \leq \Sigma_3 + \Sigma_4,
\end{align*}
where
\begin{align}\label{eq:defn of Sigma 3}
\Sigma_3 &:=\sum_{x^{1/4-\epsilon} < r \leq x^{1/4}(\log x)^{100}}\sum_{\substack{x/2 < p \leq x \\ p\equiv -1 (W) \\ p+2 \equiv 0 (r) \\ p+1 \text{ is } x^{3/10}\text{-smooth}}} \mathbf{1}_{P^-(p+2) > x^{1/4-\epsilon}}\cdot \mathbf{1}_{P^-(\frac{p+3}{2}) > (\log x)^{100}}
\end{align}
and
\begin{align}\label{eq:defn of Sigma 4}
\Sigma_4 &:= \sum_{x^{1/4-\epsilon} < r \leq x^{1/4}(\log x)^{100}}\sum_{\substack{x/2 < p \leq x \\ p\equiv -1 (W) \\ p+2 \equiv 0 (r) \\ p+1 \text{ not } x^{3/10}\text{-smooth} \\ \|\frac{\sigma_4(p+1)}{p(p+1)} + \frac{p+1}{r^4} + \frac{1}{16} \| \leq (\log x)^{-100}}} \mathbf{1}_{P^-(p+2) > x^{1/4-\epsilon}}\cdot \mathbf{1}_{P^-(\frac{p+3}{2}) >(\log x)^{100}}.
\end{align}
In later sections we prove the following propositions.

\begin{proposition}[Upper bound for $\Sigma_3$]\label{prop:sigma 3 not too big}
Let $\Sigma_3$ be defined as in \eqref{eq:defn of Sigma 3}. If $\epsilon>0$ is sufficiently small and $x$ is sufficiently large depending on $\epsilon$, then
\begin{align*}
\Sigma_3 &\leq 26 \epsilon \cdot e^{-\gamma}\frac{W^2}{\varphi(W)^3} \frac{x}{2(\log x)^2 \log ((\log x)^{100})}.
\end{align*}
\end{proposition}

\begin{proposition}[$\Sigma_4$ is small]\label{prop:sigma 4 is small}
Let $\Sigma_4$ be defined as in \eqref{eq:defn of Sigma 4}. If $\epsilon>0$ is sufficiently small and $x$ is sufficiently large depending on $\epsilon$, then
\begin{align*}
\Sigma_4 &\ll \epsilon^2\frac{W^2}{\varphi(W)^3} \frac{x}{(\log x)^2 \log\log x},
\end{align*}
where the implied constant is absolute.
\end{proposition}

The proof of Proposition \ref{prop:sigma 3 not too big} uses sieves and a theorem of Bombieri-Vinogradov-type for smooth numbers.  The use of sieves entails a loss of constant factors in the main term, but these losses are compensated by the fact that the density of the $x^\beta$-smooth numbers drops rapidly as $\beta$ decreases.

To prove Proposition \ref{prop:sigma 4 is small} we note that in $\Sigma_4$ the number $p+1$ has a prime factor $q > x^{3/10}$. Therefore, we may factor $p+1$ as before and proceed to a treatment using exponential sums. The exponential sums are again treated via the method of Weyl-van der Corput, but there are more cases to consider, and we need $p+1$ to have an additional prime factor of size $\approx x^\epsilon$ in addition to the large prime factor $q$.

\begin{proof}[Proof of Theorem \ref{thm:main theorem} assuming Propositions \ref{prop:special p implies close to integer}, \ref{prop:lower bound on number of good primes}, \ref{prop:sigma 1 is small}, \ref{prop:sigma 3 not too big}, and \ref{prop:sigma 4 is small}]
Assume for contradiction that $\alpha_4$ is rational. By Proposition \ref{prop:special p implies close to integer} we have $|\mathcal{S}| \leq \Sigma_1 + \Sigma_2 \leq \Sigma_1 + \Sigma_3 + \Sigma_4$, where $\Sigma_1, \Sigma_3$, and $\Sigma_4$ are defined in \eqref{eq:defn of Sigma 1}, \eqref{eq:defn of Sigma 3}, and \eqref{eq:defn of Sigma 4}, respectively, and $\mathcal{S}$ is defined in Proposition \ref{prop:lower bound on number of good primes}. By Propositions \ref{prop:lower bound on number of good primes} and \ref{prop:sigma 3 not too big} we have
\begin{align*}
4\epsilon \cdot e^{-\gamma}\frac{W^2}{\varphi(W)^3} \frac{x}{2(\log x)^2 \log ((\log x)^{100})}\leq \Sigma_1 + \Sigma_4.
\end{align*}
By Propositions \ref{prop:sigma 1 is small} and \ref{prop:sigma 4 is small} we have
\begin{align*}
\epsilon \frac{W^2}{\varphi(W)^3} \frac{x}{(\log x)^2 \log \log x} \ll \epsilon^2 \frac{W^2}{\varphi(W)^3} \frac{x}{(\log x)^2 \log \log x},
\end{align*}
and this is a contradiction if $\epsilon > 0$ is sufficiently small. Therefore, $\alpha_4$ is irrational.
\end{proof}

Here is the outline of the rest of the paper. In section \ref{sec:lower bound on good primes} we prove some sieve-theoretic lemmas and then use them to prove Proposition \ref{prop:lower bound on number of good primes}. In section \ref{sec:sieving shifted smooth numbers} we prove Proposition \ref{prop:sigma 3 not too big}. In section \ref{sec:sigma 1 is small} we prove Proposition \ref{prop:sigma 1 is small} with the help of a sieve-theoretic lemma and a lemma on exponential sums. In section \ref{sec:sigma 4 is small} we prove Proposition \ref{prop:sigma 4 is small}, where we require some additional lemmas on exponential sums.

\section{Sieving for shifted primes--the proof of Proposition \ref{prop:lower bound on number of good primes}}\label{sec:lower bound on good primes}

Recall the definition of the set $\mathcal{S}$ from the statement of Proposition \ref{prop:lower bound on number of good primes}. We may obtain a lower bound for $|\mathcal{S}|$ via inclusion-exclusion by writing
\begin{align}\label{eq:lower bound mathcal S}
|\mathcal{S}| &\geq \sum_{\substack{x/2 < p \leq x \\ p\equiv -1 (W) \\ P^-(p+2) > x^{1/4-\epsilon} \\ P^-(\frac{p+3}{2})> (\log x)^{100}}} 1 - \mathop{\sum\sum}_{x^{1/4-\epsilon} < r < q \leq x^{1/4}(\log x)^{100}} \sum_{\substack{x/2 < p \leq x \\ p\equiv -1 (W) \\ qr \mid p+2 \\ P^-(p+2) > x^{1/4-\epsilon} \\ P^-(\frac{p+3}{2})> (\log x)^{100}}} 1 -\sum_{\substack{x/2 < p \leq x}} \sum_{\substack{t > x^{1/4 - \epsilon} \\ t^2 \mid p+2}} 1.
\end{align}
Here the variables $p,q,r$, and $t$ denote primes. The last sum in \eqref{eq:lower bound mathcal S} is easy to bound:
\begin{align}\label{eq:bound on square part mathcal S}
\sum_{x/2 < p \leq x}\sum_{\substack{t > x^{1/4 - \epsilon} \\ t^2 \mid p+2}} 1 &= \sum_{x^{1/4-\epsilon} < t \leq 2x^{1/2}} \sum_{\substack{x/2 < p \leq x \\ p\equiv -2 (t^2)}} 1\leq \sum_{x^{1/4-\epsilon} < t \leq 2x^{1/2}}\left(\frac{x}{2t^2} + 1 \right) \ll x^{3/4 + \epsilon}.
\end{align}
The other two sums are more difficult and require somewhat delicate sieve arguments. We state and prove a few results we need, and then give the proof of Proposition \ref{prop:lower bound on number of good primes} at the end of the section.

We begin with the second sum in \eqref{eq:lower bound mathcal S}, the one involving sums over $q$ and $r$. The treatment will be representative of many sums we shall need to bound later.

\begin{lemma}[Two prime factors close together]\label{lem:two prime factors close together}
We have
\begin{align*}
\mathop{\sum\sum}_{x^{1/4-\epsilon} < r < q \leq x^{1/4}(\log x)^{100}} \sum_{\substack{x/2 < p \leq x \\ p\equiv -1 (W) \\ qr \mid p+2}} \mathbf{1}_{P^-(p+2) > x^{1/4-\epsilon}}\cdot \mathbf{1}_{P^-(\frac{p+3}{2})> (\log x)^{100}} \ll \epsilon^2 \frac{W^2}{\varphi(W)^3} \frac{x}{(\log x)^2 \log \log x}.
\end{align*}
\end{lemma}
\begin{proof}
It suffices to show
\begin{align*}
T_{q,r} = T = \sum_{\substack{x/2 < p \leq x \\ p\equiv -1 (W) \\ qr \mid p+2}} \mathbf{1}_{P^-(p+2) > x^{1/4-\epsilon}}\cdot \mathbf{1}_{P^-(\frac{p+3}{2})> (\log x)^{100}} \ll \frac{W^2}{\varphi(W)^3} \frac{x}{qr(\log x)^2 \log \log x},
\end{align*}
since by Mertens' theorem we have
\begin{align*}
\mathop{\sum\sum}_{x^{1/4-\epsilon} < r < q \leq x^{1/4}(\log x)^{100}} \frac{1}{qr} \leq \Big(\sum_{x^{1/4 - \epsilon} < q \leq x^{1/4}(\log x)^{100}} \frac{1}{q} \Big)^2 \ll \epsilon^2.
\end{align*}

We bound $T$ by using upper-bound sieves to relax the conditions that $p$ is prime and $p+2$ and $\frac{p+3}{2}$ have no small prime factors. For instance, we may write
\begin{align*}
T &\leq \sum_{\substack{x/2 < n \leq x \\ n\equiv -1 (W) \\ qr \mid n+2}} \mathbf{1}_{P^-(n) > x^{1/100}}\cdot \mathbf{1}_{P^-(n+2) > x^{1/100}}\cdot \mathbf{1}_{P^-(\frac{n+3}{2})> (\log x)^{100}}.
\end{align*}
We also utilize the fundamental lemma in order to overcome some technical issues relating to coprimality of the sieve weight variables. Let us write $D = x^{1/10}$, and $R = \lfloor (\log \log x)^2\rfloor$. Then by \eqref{eq:linear sieve inequalities} and \eqref{eq:fundamental lemma}
\begin{align*}
T &\leq \sum_{\substack{x/2 < n \leq x \\ n\equiv -1 (W) \\ n\equiv -2(qr)}} \Big(\sum_{\substack{d_1\mid P((\log x)^{100}) \\ d_1 \mid n \\ \omega(d_1) \leq 2R \\ (d_1,W)=1}} \mu(d_1) \Big)\Big(\sum_{\substack{d_2 \mid P((\log x)^{100},x^{1/100}) \\ d_2 \mid n \\ d_2 \leq D}}\lambda_{d_2}^+ \Big)\Big(\sum_{\substack{e \mid P(x^{1/100}) \\ e \mid n+2 \\ e \leq D \\(e,W)=1}}\lambda_e^+ \Big)\Big(\sum_{\substack{f\mid P((\log x)^{100}) \\ f \mid n+3 \\ \omega(f) \leq 2R \\ (f,W)=1}} \mu(f) \Big),
\end{align*}
where $\lambda_{d_2}^+,\lambda_e^+$ are the upper bound linear sieve weights. We then interchange the orders of summation, and note that $d_1,d_2,e$ and $f$ are pairwise coprime. Hence
\begin{align*}
T &\leq \sum_{\substack{e \mid P(x^{1/100}) \\ e \leq D \\(e,W)=1}}\lambda_e^+\sum_{\substack{d_1\mid P((\log x)^{100}) \\ \omega(d_1) \leq 2R \\ (d_1,eW)=1}} \mu(d_1)\sum_{\substack{f\mid P((\log x)^{100}) \\ \omega(f) \leq 2R \\ (f,d_1eW)=1}} \mu(f)\sum_{\substack{d_2 \mid P((\log x)^{100},x^{1/100}) \\ d_2 \leq D \\(d_2,d_1ef)=1}}\lambda_{d_2}^+\sum_{\substack{x/2 < n \leq x \\ n\equiv v (d_1d_2efqrW)}} 1,
\end{align*}
where we have used the Chinese remainder theorem to write the congruence conditions on $n$ as a single congruence condition modulo $d_1d_2efqrW$. Since $d_1d_2efqrW \ll x^{3/4}$, say, we have
\begin{align*}
T &\leq \frac{x}{2Wqr}\sum_{\substack{e \mid P(x^{1/100}) \\ e \leq D \\(e,W)=1}}\frac{\lambda_e^+}{e}\sum_{\substack{d_1\mid P((\log x)^{100}) \\ \omega(d_1) \leq 2R \\ (d_1,eW)=1}} \frac{\mu(d_1)}{d_1}\sum_{\substack{f\mid P((\log x)^{100}) \\ \omega(f) \leq 2R \\ (f,d_1eW)=1}} \frac{\mu(f)}{f}\sum_{\substack{d_2 \mid P((\log x)^{100},x^{1/100}) \\ d_2 \leq D \\(d_2,d_1ef)=1}}\frac{\lambda_{d_2}^+}{d_2} + O\left(\frac{x^{1-\epsilon}}{qr}\right).
\end{align*}
To finish the proof of the lemma, then, it suffices to show that this fourfold sum over $d_1,d_2,e$, and $f$ is $\ll \frac{W^3}{\varphi(W)^3} \frac{1}{(\log x)^2 \log \log x}$.

Since every prime divisor of $d_2$ is $> (\log x)^{100}$, we may remove the condition that $d_2$ is coprime to $d_1ef$ at the cost of a negligible error:
\begin{align*}
\sum_{\substack{d_2 \mid P((\log x)^{100},x^{1/100}) \\ d_2 \leq D \\(d_2,d_1ef)>1}}\frac{1}{d_2} &\leq \sum_{d_2 \leq D}\frac{1}{d_2}\sum_{\substack{p\mid d_1ef \\ p \mid d_2 \\ p > (\log x)^{100}}}1 \leq \sum_{\substack{p \mid d_1ef \\ p > (\log x)^{100}}} \sum_{\substack{d_2\leq D \\ p\mid d_2}} \frac{1}{d} \ll \log x\sum_{\substack{p\mid d_1ef \\ p > (\log x)^{100}}} \frac{1}{p} \\
&\ll (\log x)^{-98}.
\end{align*}
Therefore we may make the sum over $d_2$ independent of the other sums, and by Lemma \ref{lem:linear sieve} we have
\begin{align*}
\sum_{\substack{d_2 \mid P((\log x)^{100},x^{1/100}) \\ d_2 \leq D}}\frac{\lambda_{d_2}^+}{d_2} \ll \prod_{(\log x)^{100} < p \leq x^{1/100}} \left(1 - \frac{1}{p} \right) \ll \frac{\log \log x}{\log x}.
\end{align*}
It now suffices to show that
\begin{align*}
\sum_{\substack{e \mid P(x^{1/100}) \\ e \leq D \\(e,W)=1}}\frac{\lambda_e^+}{e}\sum_{\substack{d_1\mid P((\log x)^{100}) \\ \omega(d_1) \leq 2R \\ (d_1,eW)=1}} \frac{\mu(d_1)}{d_1}\sum_{\substack{f\mid P((\log x)^{100}) \\ \omega(f) \leq 2R \\ (f,d_1eW)=1}} \frac{\mu(f)}{f} \ll \frac{W^3}{\varphi(W)^3}\frac{1}{(\log x)(\log \log x)^2}.
\end{align*}

We may treat the sums over $f$ and $d_1$ quite accurately. By \cite[(6.8)]{FI2010} we have
\begin{align*}
\sum_{\substack{f\mid P((\log x)^{100}) \\ \omega(f) \leq 2R \\ (f,d_1eW)=1}} \frac{\mu(f)}{f} &= \sum_{\substack{f\mid P((\log x)^{100}) \\ (f,d_1eW)=1}} \frac{\mu(f)}{f} + O \Big(\sum_{\substack{f\mid P((\log x)^{100}) \\ \omega(f) = 2R+1}} \frac{\mu^2(f)}{f} \Big),
\end{align*}
and the error term has size $\leq \frac{1}{(2R+1)!}\Big(\sum_{p\leq (\log x)^{100}} \frac{1}{p} \Big)^{2R+1} \ll_A (\log x)^{-A}$. By multiplicativity
\begin{align*}
\sum_{\substack{f\mid P((\log x)^{100}) \\ (f,d_1eW)=1}} \frac{\mu(f)}{f} &= \frac{d_1W}{\varphi(d_1)\varphi(W)} \prod_{\substack{p \mid e \\ p\leq (\log x)^{100}}}\left( 1 - \frac{1}{p}\right)^{-1} \prod_{p \leq (\log x)^{100}} \left( 1 - \frac{1}{p}\right).
\end{align*}
We can remove the condition $p\leq (\log x)^{100}$ on the primes dividing $e$ at the cost of an acceptable error term. By Mertens' theorem the product over primes $p\leq (\log x)^{100}$ has size $\asymp \frac{1}{\log \log x}$, so we must show
\begin{align*}
\sum_{\substack{e \mid P(x^{1/100}) \\ e \leq D \\(e,W)=1}}\frac{\lambda_e^+}{\varphi(e)}\sum_{\substack{d_1\mid P((\log x)^{100}) \\ \omega(d_1) \leq 2R \\ (d_1,eW)=1}} \frac{\mu(d_1)}{\varphi(d_1)} \ll \frac{W^2}{\varphi(W)^2} \frac{1}{(\log x) (\log \log x)}.
\end{align*}

By a similar argument to the one we did for the sum over $f$, the sum over $d_1$ is equal to
\begin{align}\label{eq:intermed sieve quant}
\prod_{p\mid e}\frac{p-1}{p-2} \prod_{D_0 < p \leq (\log x)^{100}}\left(1 - \frac{1}{p-1}\right)
\end{align}
and an error of size $O((\log x)^{-50})$, say. If we write 
\begin{align}\label{eq:defn of h}
h(e) &= \prod_{p\mid e}\frac{p-1}{p-2}
\end{align}
for the multiplicative function, we see by multiplying and dividing that the quantity in \eqref{eq:intermed sieve quant} is
\begin{align*}
h(e) \prod_{D_0 < p \leq (\log x)^{100}}\left(1 - \frac{1}{(p-1)^2}\right) \prod_{D_0 < p\leq (\log x)^{100}}\left( 1 - \frac{1}{p}\right).
\end{align*}
Excluding the term $h(e)$, the products over primes have size $\asymp \frac{W}{\varphi(W)} \frac{1}{\log \log x}$. Hence, in order to complete the proof it suffices to show (at last!) that
\begin{align*}
\sum_{\substack{e \mid P(x^{1/100}) \\ e \leq D \\(e,W)=1}}\frac{\lambda_e^+h(e)}{\varphi(e)} \ll \frac{W}{\varphi(W)} \frac{1}{\log x}.
\end{align*}

By Lemma \ref{lem:linear sieve}, we have
\begin{align*}
\sum_{\substack{e \mid P(x^{1/100}) \\ e \leq D \\(e,W)=1}}\frac{\lambda_e^+h(e)}{\varphi(e)} &\ll \prod_{D_0 < p\leq x^{1/100}} \left(1 - \frac{h(p)}{p-1}\right)=\prod_{D_0 < p\leq x^{1/100}} \left(1 - \frac{1}{p-2}\right) \\ 
&= \prod_{D_0 < p\leq x^{1/100}}\left(1 - \frac{2}{p^2-3p+2} \right)\prod_{D_0 < p\leq (\log x)^{100}}\left(1 - \frac{1}{p}\right) \\
&\ll \frac{W}{\varphi(W)}\frac{1}{\log x},
\end{align*}
as required.
\end{proof}

To prove a lower bound on the first sum in \eqref{eq:lower bound mathcal S} we need a familiar inequality related to the vector sieve \cite{BF1994, BF1996}.

\begin{lemma}[Vector sieve inequality]\label{lem:vector sieve inequality}
Let $\delta_1,\delta_2$ be non-negative real numbers. For $i\in \{1,2\}$ let $\delta_i^\pm$ be real numbers satisfying $\delta_i^- \leq \delta_i \leq \delta_i^+$. Then $\delta_1 \delta_2 \geq \delta_1^+ \delta_2^- + \delta_1^-\delta_2^+ - \delta_1^+ \delta_2^+$.
\end{lemma}
\begin{proof}
By positivity we have $(\delta_1^+ - \delta_1)(\delta_2^+ - \delta_2) \geq 0$. Rearranging yields
\begin{align*}
\delta_2 \delta_2 &\geq \delta_1^+ \delta_2 + \delta_1 \delta_2^+ - \delta_1^+ \delta_2^+ \geq \delta_1^+ \delta_2^- + \delta_1^-\delta_2^+ - \delta_1^+ \delta_2^+,
\end{align*}
where in the second inequality we have used $\delta_i^+ \geq 0$.
\end{proof}

\begin{lemma}[Shifted primes with no small prime factors]\label{lem:lower bound on shifted prime}
For $x$ sufficiently large we have
\begin{align*}
\sum_{\substack{x/2 < p \leq x \\ p\equiv -1 (W) \\ P^-(p+2) > x^{1/4-\epsilon} \\ P^-(\frac{p+3}{2})> (\log x)^{100}}} 1 &\geq 31 \epsilon \cdot e^{-\gamma}\frac{W^2}{\varphi(W)^3} \frac{x}{2(\log x)^2 \log ((\log x)^{100})}.
\end{align*}
\end{lemma}
\begin{proof}
We write $D = x^{1/2}\exp(-\sqrt{\log x})$ and $R = \lfloor (\log \log x)^2\rfloor$. For $p\equiv -1 \pmod{W}$ we have by Lemma \ref{lem:vector sieve inequality}
\begin{align*}
&\mathbf{1}_{P^-(p+2) > x^{1/4-\epsilon}} \mathbf{1}_{P^-(\frac{p+3}{2})> (\log x)^{100}}  \\ 
&\geq \Big(\sum_{\substack{d\mid P(x^{1/4-\epsilon}) \\ d\leq D \\ d \mid p+2 \\ (d,W)=1}} \lambda_d^+ \Big)\Big(\sum_{\substack{e \mid P((\log x)^{100}) \\ \omega(e) \leq 2R-1 \\ e\mid p+3 \\ (e,W)=1}}\mu(e) \Big) + \Big(\sum_{\substack{d\mid P(x^{1/4-\epsilon}) \\ d\leq D \\ d \mid p+2 \\ (d,W)=1}} \lambda_d^- \Big)\Big(\sum_{\substack{e \mid P((\log x)^{100}) \\ \omega(e) \leq 2R \\ e\mid p+3 \\ (e,W)=1}}\mu(e) \Big) \\
&- \Big(\sum_{\substack{d\mid P(x^{1/4-\epsilon}) \\ d\leq D \\ d \mid p+2 \\ (d,W)=1}} \lambda_d^+ \Big)\Big(\sum_{\substack{e \mid P((\log x)^{100}) \\ \omega(e) \leq 2R \\ e\mid p+3 \\ (e,W)=1}}\mu(e) \Big).
\end{align*}
Observe that
\begin{align*}
&\Big(\sum_{\substack{d\mid P(x^{1/4-\epsilon}) \\ d\leq D \\ d \mid p+2 \\ (d,W)=1}} \lambda_d^+ \Big)\Big(\sum_{\substack{e \mid P((\log x)^{100}) \\ \omega(e) \leq 2R-1 \\ (e,W)=1}}\mu(e) \Big) - \Big(\sum_{\substack{d\mid P(x^{1/4-\epsilon}) \\ d\leq D \\ d \mid p+2 \\ (d,W)=1}} \lambda_d^+ \Big)\Big(\sum_{\substack{e \mid P((\log x)^{100}) \\ \omega(e) \leq 2R \\ (e,W)=1}}\mu(e) \Big) \\
&= -\Big(\sum_{\substack{d\mid P(x^{1/4-\epsilon}) \\ d\leq D \\ d \mid p+2 \\ (d,W)=1}} \lambda_d^+ \Big)\Big(\sum_{\substack{e \mid P((\log x)^{100}) \\ \omega(e) =2R \\ (e,W)=1}}\mu(e) \Big),
\end{align*}
and
\begin{align*}
\sum_{\substack{x/2 < p \leq x \\ p\equiv - 1(W)}} \Big(\sum_{\substack{d\mid P(x^{1/4-\epsilon}) \\ d\leq D \\ d \mid p+2 \\ (d,W)=1}} \lambda_d^+ \Big)\Big(\sum_{\substack{e \mid P((\log x)^{100}) \\ \omega(e) =2R \\ (e,W)=1}}\mu(e) \Big) &\leq \sum_{d\leq D} \sum_{\substack{e\mid P((\log x)^{100}) \\ \omega(e)=2R \\ (e,d)=1}} \mu^2(e)\sum_{\substack{x/2 < n \leq x \\ n\equiv -2 (d) \\ n\equiv -3(e)}} 1.
\end{align*}
Since $d$ and $e$ are coprime we may combine the congruence conditions into a single congruence modulo $de$, and since $e\leq \exp(200(\log \log x)^3)$ this is
\begin{align*}
&\leq x\sum_{d\leq D}\frac{1}{d} \sum_{\substack{e\mid P((\log x)^{100}) \\ \omega(e)=2R}} \frac{\mu^2(e)}{e} + O(x^{1/2}).
\end{align*}
This has size $\ll_A x(\log x)^{-A}$ because
\begin{align*}
\sum_{\substack{e\mid P((\log x)^{100}) \\ \omega(e)=2R}} \frac{\mu^2(e)}{e} &\leq \frac{1}{(2R)!}\Big(\sum_{p\leq (\log x)^{100}} \frac{1}{p} \Big)^{2R} \ll_A (\log x)^{-A}.
\end{align*}

We have therefore proved
\begin{align}\label{eq:lower bound after vector sieve}
\sum_{\substack{x/2 < p \leq x \\ p\equiv -1 (W) \\ P^-(p+2) > x^{1/4-\epsilon} \\ P^-(\frac{p+3}{2})> (\log x)^{100}}} 1 &\geq \sum_{\substack{x/2 < p \leq x \\ p\equiv -1(W)}} \Big(\sum_{\substack{d\mid P(x^{1/4-\epsilon}) \\ d\leq D \\ d \mid p+2 \\ (d,W)=1}} \lambda_d^- \Big)\Big(\sum_{\substack{e \mid P((\log x)^{100}) \\ \omega(e) \leq 2R \\ e\mid p+3 \\ (e,W)=1}}\mu(e) \Big) +O_A\left( \frac{x}{(\log x)^A}\right).
\end{align}
Let us write
\begin{align*}
U &:= \sum_{\substack{x/2 < p \leq x \\ p\equiv -1(W)}} \Big(\sum_{\substack{d\mid P(x^{1/4-\epsilon}) \\ d\leq D \\ d \mid p+2 \\ (d,W)=1}} \lambda_d^- \Big)\Big(\sum_{\substack{e \mid P((\log x)^{100}) \\ \omega(e) \leq 2R \\ e\mid p+3 \\ (e,W)=1}}\mu(e) \Big) = \sum_{\substack{d\mid P(x^{1/4-\epsilon}) \\ d\leq D \\ (d,W)=1}}\lambda_d^- \sum_{\substack{e\mid (P(\log x)^{100}) \\ \omega(e)\leq 2R \\ (e,dW)=1}}\mu(e)\sum_{\substack{x/2 < p \leq x \\ p\equiv -1 (W) \\ p\equiv -2 (d) \\ p\equiv -3 (e)}} 1.
\end{align*}
By the Chinese remainder theorem we may combine the congruence conditions on $p$ into a single primitive residue class modulo $deW$. By inclusion-exclusion we then have
\begin{align*}
U &= \frac{\pi(x)-\pi(x/2)}{\varphi(W)}\sum_{\substack{d\mid P(x^{1/4-\epsilon}) \\ d\leq D \\ (d,W)=1}}\frac{\lambda_d^-}{\varphi(d)} \sum_{\substack{e\mid P((\log x)^{100}) \\ \omega(e)\leq 2R \\ (e,dW)=1}}\frac{\mu(e)}{\varphi(e)} + O(E),
\end{align*}
where
\begin{align*}
E &\ll \max_{y \in \{x/2,x\}}\sum_{d\leq D}\sum_{e\leq (\log x)^{200R}} \max_{(a,deW)=1}\Big|\sum_{\substack{p\leq y \\ p\equiv a (deW)}}1 - \frac{1}{\varphi(deW)}\sum_{\substack{p\leq y \\ (p,deW)=1}} 1\Big| \\
&\leq \max_{y \in \{x/2,x\}} \sum_{g\leq x^{1/2}\exp(-(\log x)^{1/3})} \tau(g) \max_{(a,g)=1}\Big|\sum_{\substack{p\leq y \\ p\equiv a (g)}}1 - \frac{1}{\varphi(g)}\sum_{\substack{p\leq y \\ (p,g)=1}} 1\Big|.
\end{align*}
By Cauchy-Schwarz, the trivial bound
\begin{align*}
\max_{(a,g)=1}\Big|\sum_{\substack{p\leq y \\ p\equiv a (g)}}1 - \frac{1}{\varphi(g)}\sum_{\substack{p\leq y \\ (p,g)=1}} 1\Big| \ll \frac{x}{\varphi(g)},
\end{align*}
and the Bombieri-Vinogradov theorem (\cite[Chapter 28]{Dav2000} or \cite[Theorem 9.18]{FI2010}) we obtain $E \ll x(\log x)^{-100}$, say. Since
\begin{align*}
\sum_{\substack{e\mid P((\log x)^{100}) \\ \omega(e)\leq 2R \\ (e,dW)=1}}\frac{\mu(e)}{\varphi(e)} = \sum_{\substack{e\mid P((\log x)^{100}) \\ (e,dW)=1}}\frac{\mu(e)}{\varphi(e)} + O_A((\log x)^{-A})
\end{align*}
we find
\begin{align*}
U &= \frac{\pi(x)-\pi(x/2)}{\varphi(W)}\sum_{\substack{d\mid P(x^{1/4-\epsilon}) \\ d\leq D \\ (d,W)=1}}\frac{\lambda_d^-}{\varphi(d)} \prod_{\substack{D_0 < p\leq (\log x)^{100} \\ p\nmid d}}\left(1 - \frac{1}{p-1} \right) + O(x(\log x)^{-100}).
\end{align*}
By trivial estimation we have
\begin{align*}
&\prod_{\substack{D_0 < p\leq (\log x)^{100} \\ p\nmid d}}\left(1 - \frac{1}{p-1} \right) = \prod_{p\mid d}\left(1 - \frac{1}{p-1} \right)^{-1} \prod_{D_0 < p\leq (\log x)^{100}} \left(1 - \frac{1}{p-1}\right) + O((\log x)^{-75}) \\
&= h(d)\prod_{D_0 < p\leq (\log x)^{100}}\left(1 - \frac{1}{(p-1)^2} \right) \frac{W}{\varphi(W)}\prod_{p\leq (\log x)^{100}}\left(1 - \frac{1}{p}\right) + O((\log x)^{-75}),
\end{align*}
where $h(d)$ is defined in \eqref{eq:defn of h}.

It follows that
\begin{align*}
U &= \frac{W}{\varphi(W)^2}(\pi(x)-\pi(x/2))\prod_{D_0 < p\leq (\log x)^{100}}\left(1 - \frac{1}{(p-1)^2} \right)\prod_{p\leq (\log x)^{100}}\left(1 - \frac{1}{p}\right)\sum_{\substack{d\mid P(x^{1/4-\epsilon}) \\ d\leq D \\ (d,W)=1}}\frac{\lambda_d^- h(d)}{\varphi(d)} \\ 
&+ O(x (\log x)^{-20}).
\end{align*}
Lemma \ref{lem:linear sieve} yields
\begin{align*}
\sum_{\substack{d\mid P(x^{1/4-\epsilon}) \\ d\leq D \\ (d,W)=1}}\frac{\lambda_d^- h(d)}{\varphi(d)} &\geq (1-O(\epsilon)) f\left(\frac{\log D}{\log (x^{1/4-\epsilon})} \right) \prod_{D_0 < p\leq x^{1/4-\epsilon}}\left( 1 - \frac{1}{p-2}\right).
\end{align*}
Since
\begin{align*}
\prod_{D_0 < p\leq x^{1/4-\epsilon}}\left( 1 - \frac{1}{p-2}\right) = \frac{W}{\varphi(W)}\prod_{D_0 < p \leq x^{1/4-\epsilon}}\left(1 - \frac{2}{p^2-3p+2} \right)\prod_{p\leq x^{1/4-\epsilon}}\left(1 - \frac{1}{p}\right)
\end{align*}
we obtain
\begin{align*}
U &\geq (1-O(\epsilon))f\left(\frac{\log D}{\log (x^{1/4-\epsilon})} \right) \frac{W^2}{\varphi(W)^3}\frac{x}{2\log x}\prod_{p\leq (\log x)^{100}}\left(1 - \frac{1}{p}\right) \prod_{p\leq x^{1/4-\epsilon}}\left(1 - \frac{1}{p}\right).
\end{align*}
By Mertens' theorem this simplifies to
\begin{align*}
U &\geq (1-O(\epsilon)) \frac{e^{-2\gamma}}{\frac{1}{4}-\epsilon}f\left(\frac{\log D}{\log (x^{1/4-\epsilon})} \right) \frac{W^2}{\varphi(W)^3}\frac{x}{2(\log x)^2\log ((\log x)^{100})}.
\end{align*}

The function $f(s)$ is equal to $\frac{2e^\gamma}{s}\log(s-1)$ for $2\leq s\leq 4$, and
\begin{align*}
\frac{\log D}{\log (x^{1/4-\epsilon})} = \frac{2}{1-4\epsilon} + O((\log x)^{-1/2}),
\end{align*}
so by continuity
\begin{align*}
f\left(\frac{\log D}{\log (x^{1/4-\epsilon})} \right) &\geq (1-O(\epsilon))e^\gamma \log \left(\frac{2}{1-4\epsilon} - 1 \right) = (1-O(\epsilon))e^\gamma\log \left( \frac{1+4\epsilon}{1-4\epsilon}\right) \\
&\geq (1-O(\epsilon))8\epsilon \cdot e^\gamma,
\end{align*}
and therefore
\begin{align*}
U &\geq (1-O(\epsilon)) 32 \epsilon \cdot e^{-\gamma}\frac{W^2}{\varphi(W)^3} \frac{x}{2(\log x)^2 \log ((\log x)^{100})}.
\end{align*}
Since $\epsilon>0$ is sufficiently small we finish by reference to \eqref{eq:lower bound after vector sieve}.
\end{proof}

\begin{proof}[Proof of Proposition \ref{prop:lower bound on number of good primes}]
By \eqref{eq:lower bound mathcal S}, \eqref{eq:bound on square part mathcal S}, Lemma \ref{lem:lower bound on shifted prime}, and Lemma \ref{lem:two prime factors close together} we have
\begin{align*}
|\mathcal{S} | &\geq 31 \epsilon \cdot e^{-\gamma}\frac{W^2}{\varphi(W)^3} \frac{x}{2(\log x)^2 \log ((\log x)^{100})} - O\left(\epsilon^2\frac{W^2}{\varphi(W)^3} \frac{x}{(\log x)^2 \log \log x} \right) - O(x^{3/4+\epsilon}) \\
&\geq 30 \epsilon \cdot e^{-\gamma}\frac{W^2}{\varphi(W)^3} \frac{x}{2(\log x)^2 \log ((\log x)^{100})},
\end{align*}
the last inequality following if $\epsilon>0$ is sufficiently small and $x$ is sufficiently large depending on $\epsilon$.
\end{proof}

\section{Sieving shifted smooth numbers--the proof of Proposition \ref{prop:sigma 3 not too big}}\label{sec:sieving shifted smooth numbers}

We write $y = x^{3/10}$. Then the sum $\Sigma_3$ in \eqref{eq:defn of Sigma 3} which appears in the statement of Proposition \ref{prop:sigma 3 not too big} may be written as
\begin{align*}
\Sigma_3 &=\sum_{x^{1/4-\epsilon} < r \leq x^{1/4}(\log x)^{100}}\sum_{\substack{x/2 < p \leq x \\ p\equiv -1 (W) \\ p+2 \equiv 0 (r) \\ p+1 \text{ is } y\text{-smooth}}} \mathbf{1}_{P^-(p+2) > x^{1/4-\epsilon}}\cdot \mathbf{1}_{P^-(\frac{p+3}{2}) > (\log x)^{100}}.
\end{align*}

\begin{proof}[Proof of Proposition \ref{prop:sigma 3 not too big}]
In order to prove Proposition \ref{prop:sigma 3 not too big}, we must show
\begin{align*}
\Sigma_3 &\leq 26 \epsilon \cdot e^{-\gamma}\frac{W^2}{\varphi(W)^3} \frac{x}{2(\log x)^2 \log ((\log x)^{100})}.
\end{align*}
We change variables $p+1 =Wn$ and observe
\begin{align*}
\Sigma_3 &\leq O(1)+\sum_{x^{1/4-\epsilon} < r \leq x^{1/4}(\log x)^{100}}\sum_{\substack{x/2W < n \leq x/W \\ P^-(Wn-1)>x^{1/2-\epsilon} \\ Wn+1 \equiv 0 (r) \\ n \text{ is } y\text{-smooth}}} \mathbf{1}_{P^-(Wn+1) > x^{1/4-\epsilon}}\cdot \mathbf{1}_{P^-(\frac{Wn+2}{2}) > (\log x)^{100}}.
\end{align*}
We use sieves to control the conditions on the prime factors of $Wn-1, Wn+1$, and $\frac{Wn+2}{2}$. In considering $y$-smooth numbers in arithmetic progressions we are limited to moduli of size at most $x^{1/2-\epsilon}$, just as for the primes, and this requires us to relax some of the conditions:
\begin{align*}
\Sigma_3 &\leq O(1)+\sum_{x^{1/4-\epsilon} < r \leq x^{1/4}(\log x)^{100}}\sum_{\substack{x/2W < n \leq x/W \\ P^-(Wn-1)> x^{1/8-2\epsilon} \\ Wn+1 \equiv 0 (r) \\ n \text{ is } y\text{-smooth}}} \mathbf{1}_{P^-(Wn+1) > x^{1/8-2\epsilon}}\cdot \mathbf{1}_{P^-(\frac{Wn+2}{2}) > (\log x)^{100}}.
\end{align*}

As we have done in the proof of Lemma \ref{lem:two prime factors close together}, we apply linear upper-bound sieves to the sum, and split one of the conditions so we use the fundamental lemma for the small primes. We therefore have
\begin{align*}
\Sigma_3 -O(1) \leq \sum_{x^{1/4-\epsilon} < r \leq x^{1/4}(\log x)^{100}}&\sum_{\substack{x/2W < n \leq x/W \\ Wn+1 \equiv 0 (r) \\ n \text{ is } y\text{-smooth}}} \Big(\sum_{\substack{d_1 \mid P((\log x)^{100}) \\ \omega(d_1) \leq 2R \\ d_1 \mid Wn-1 \\ (d_1,W)=1}}\mu(d_1) \Big) \Big(\sum_{\substack{d_2 \mid P((\log x)^{100},x^{1/8-2\epsilon}) \\ d_2\leq D \\ d_2 \mid Wn-1}}\lambda_{d_2}^+ \Big) \\
&\times \Big(\sum_{\substack{e \mid P(x^{1/8-2\epsilon}) \\ e\leq D \\ e \mid Wn+1 \\ (e,W)=1}}\lambda_{e}^+ \Big)\Big(\sum_{\substack{f \mid P((\log x)^{100}) \\ \omega(f) \leq 2R \\ f \mid Wn+2 \\ (f,W)=1}}\mu(f) \Big),
\end{align*}
where $R = \lfloor (\log \log x)^2 \rfloor$, $D = x^{1/8-\epsilon}$, and $\lambda_{d_2}^+,\lambda_e^+$ are the upper-bound linear sieve weights. Let $V$ denote the sum on the right-hand side. We interchange the order of summation to obtain
\begin{align*}
V = \sum_{x^{1/4-\epsilon} < r \leq x^{1/4}(\log x)^{100}} &\sum_{\substack{d_1 \mid P((\log x)^{100}) \\ \omega(d_1) \leq 2R \\ (d_1,W)=1}} \mu(d_1) \sum_{\substack{d_2 \mid P((\log x)^{100}, x^{1/8-2\epsilon}) \\ d_2 \leq D}} \lambda_{d_2}^+\sum_{\substack{e\mid P(x^{1/8-2\epsilon}) \\ e\leq D \\ (e,d_1d_2W)=1}}\lambda_e^+ \\
&\times \sum_{\substack{f \mid P((\log x)^{100}) \\ \omega(f) \leq 2R \\ (f,d_1d_2eW)=1}} \mu(f)\sum_{\substack{x/2W < n \leq x/W \\ n \text{ is } y\text{-smooth} \\m\equiv v (d_1d_2efr)}} 1,
\end{align*}
where we used the Chinese remainder theorem to combine the congruence conditions on $n$ into a single primitive congruence class. By inclusion-exclusion and combining variables we then have
\begin{align*}
V &=\sum_{x^{1/4-\epsilon} < r \leq x^{1/4}(\log x)^{100}} \frac{1}{\varphi(r)}\sum_{\substack{e\mid P(x^{1/8-2\epsilon}) \\ e\leq D \\ (e,W)=1}}\frac{\lambda_e^+}{\varphi(e)}\sum_{\substack{d_2 \mid P((\log x)^{100}, x^{1/8-2\epsilon}) \\ d_2 \leq D \\ (d_2,e)=1}} \frac{\lambda_{d_2}^+}{\varphi(d_2)} \\
&\times \sum_{\substack{d_1 \mid P((\log x)^{100}) \\ \omega(d_1) \leq 2R \\ (d_1,d_2eW)=1}} \frac{\mu(d_1)}{\varphi(d_1)}  \sum_{\substack{f \mid P((\log x)^{100}) \\ \omega(f) \leq 2R \\ (f,d_1d_2eW)=1}} \frac{\mu(f)}{\varphi(f)} \sum_{\substack{x/2W < n \leq x/W \\ n \text{ is } y\text{-smooth} \\ (n,rd_1d_2ef)=1}} 1 + O(E),
\end{align*}
where
\begin{align*}
E &= \max_{z \in \{x/2W,x/W\}}\sum_{g\leq x^{1/2-\epsilon}}  \tau(g)^4\max_{(a,g)=1} \Big| \sum_{\substack{n\leq z \\ n \text{ is } y\text{-smooth} \\ n\equiv a (g)}} 1- \frac{1}{\varphi(g)}\sum_{\substack{n\leq z \\ n \text{ is } y\text{-smooth} \\ (n,g)=1}} 1 \Big|.
\end{align*}
By Cauchy-Schwarz, the trivial bound
\begin{align*}
\Big| \sum_{\substack{n\leq z \\ n \text{ is } y\text{-smooth} \\ n\equiv a (g)}} 1- \frac{1}{\varphi(g)}\sum_{\substack{n\leq z \\ n \text{ is } y\text{-smooth} \\ (m,g)=1}} 1 \Big| \ll \frac{z}{\varphi(g)},
\end{align*}
and the Bombieri-Vinogradov theorem for $y$-smooth numbers \cite[Corollaire 1]{FT1996} we obtain
\begin{align*}
E \ll \frac{x}{(\log x)^{100}}.
\end{align*}

The condition $(n,r)=1$ may be removed at the cost of a negligible error. We note that the sum over $r$ satisfies
\begin{align*}
\sum_{x^{1/4-\epsilon} < r \leq x^{1/4}(\log x)^{100}} \frac{1}{\varphi(r)} &= O(x^{-1/4+\epsilon}) + \sum_{x^{1/4-\epsilon} < r \leq x^{1/4}(\log x)^{100}} \frac{1}{r} \\ 
&= -\log\left(1-4\epsilon \right) + O\left( \frac{\log \log x}{\log x}\right)\leq (1+O(\epsilon))4\epsilon.
\end{align*}
We rearrange to make the sum over $n$ the outermost summation. We then successively evaluate the sums over $f$, then $d_1$, via arguments and estimations like those in the proofs of Lemmas \ref{lem:two prime factors close together} and \ref{lem:lower bound on shifted prime}. We can remove the condition that $d_2$ and $e$ are coprime since $d_2$ is only supported on integers with prime factors $> (\log x)^{100}$. The sums over $e$ and $d_2$ are then independent. The condition that $d_2$ and $n$ are coprime may similarly be removed. We use Lemma \ref{lem:linear sieve} to get an upper bound for the sum over $d_2$ and obtain
\begin{align*}
V &\leq (1+O(\epsilon))4\epsilon \frac{W^2}{\varphi(W)^2} \prod_{p > D_0 }(1+O(p^{-2})) \prod_{p\leq (\log x)^{100}} \left(1 - \frac{1}{p} \right) F \left(\frac{\log D}{\log (x^{1/8-2\epsilon})} \right)  \\
&\times \prod_{p\leq x^{1/8-2\epsilon}} \left( 1 - \frac{1}{p}\right)\sum_{\substack{x/2W < n \leq x/W \\ n \text{ is } y\text{-smooth}}} j(n) \sum_{\substack{e\mid P(x^{1/8-2\epsilon}) \\ e\leq D \\ (e,nW)=1}}\frac{\lambda_e^+}{\varphi(e)}j(e),
\end{align*}
where $j(n)$ is the multiplicative function given by 
\begin{align*}
j(n) = \prod_{\substack{p\mid n \\ p > D_0}} \frac{p-1}{p-3}.
\end{align*}
Using Lemma \ref{lem:linear sieve} again to evaluate the sum over $e$ and doing a bit of easy estimation, we obtain
\begin{align*}
V &\leq (1+O(\epsilon))4\epsilon \frac{W^3}{\varphi(W)^3} \frac{e^{-\gamma}}{\log((\log x)^{100})} \frac{\left(\frac{1}{8}-2\epsilon \right)^{-2}}{(\log x)^2}e^{-2\gamma}F \left(\frac{\log D}{\log (x^{1/8-2\epsilon})} \right)^2\sum_{\substack{x/2W < n \leq x/W \\ n \text{ is } y\text{-smooth}}} k(n),
\end{align*}
where $k(n)$ is the multiplicative function given by $k(n) = \prod_{\substack{p\mid n \\ p > D_0}} \frac{p-1}{p-4}$. The function $F(s)$ is given by $F(s) = 2e^{\gamma}s^{-1}$ for $1 \leq s \leq 3$, and therefore
\begin{align*}
V &\leq (1+O(\epsilon))1024\epsilon \frac{W^3}{\varphi(W)^3} \frac{e^{-\gamma}}{(\log x)^2\log((\log x)^{100})}\sum_{\substack{x/2W < n \leq x/W \\ n \text{ is } y\text{-smooth}}} k(n).
\end{align*}

Since $\frac{p-1}{p-4} = 1 + \frac{3}{p-4} = 1 + \frac{3}{p} + O(p^{-2})$ we see that
\begin{align*}
k(n) &\leq (1+O(D_0^{-1})) \prod_{\substack{p\mid n \\ p > D_0}}\left(1 + \frac{3}{p}\right) = (1+O(D_0^{-1})) \sum_{\substack{d\mid n \\ (d,W)=1}}\frac{\mu^2(d)3^{\omega(d)}}{d}.
\end{align*}
Therefore, by interchanging the order of summation we deduce
\begin{align*}
\sum_{\substack{x/2W < n \leq x/W \\ n \text{ is } y\text{-smooth}}} k(n) &\leq (1+O(D_0^{-1}))\sum_{(d,W)=1} \frac{\mu^2(d)3^{\omega(d)}}{d}\sum_{\substack{x/2dW < n \leq x/dW \\ n \text{ is } y\text{-smooth}}}  1
\end{align*}
by a change of variables. The contribution from $d > (\log x)^{10}$ is acceptably small by trivial estimation, so we may assume $d\leq (\log x)^{10}$.

Classical results \cite[Theorem 1]{Hil1986} on counts for smooth numbers yield
\begin{align*}
\Psi(x,y) = x\rho(u)\left(1 + O \left( \frac{\log(u+1)}{\log y}\right) \right),
\end{align*}
with $u = \frac{\log x}{\log y}$ provided $y\geq 2$ and $u\leq (\log x)^{1/10}$, say. In our case we have $u = O(1)$, so we may write $\Psi(x,y) = x\rho(u) + O(x/\log x)$. We apply this with $x \rightarrow x/dW$ and $x \rightarrow x/2dW$ and subtract to get an estimate for $\sum_{\substack{x/2dW < n \leq x/dW \\ n \text{ is } y\text{-smooth}}}  1$. The error term contributes an acceptably small amount, since the sum $\sum_{d\geq 1}\frac{\mu^2(d)3^{\omega(d)}}{d^2}$ is convergent. Using the differential delay equation \eqref{eq:dickman rho differential delay} for $\rho$ we find
\begin{align*}
\rho\Big(\frac{\log(x/dW)}{\log y} \Big) = \rho \Big( \frac{\log x}{\log y}\Big) + O\Big(\frac{\log \log x}{\log x} \Big),
\end{align*}
and similarly with $x/dW$ replaced by $x/2dW$.

We have therefore shown
\begin{align*}
\sum_{\substack{x/2W < n \leq x/W \\ n \text{ is } y\text{-smooth}}} k(n) &\leq (1+O(\epsilon))\frac{x}{2W}\rho \left( \frac{\log x}{\log y}\right)\sum_{(d,W)=1} \frac{\mu^2(d)3^{\omega(d)}}{d^2}\leq (1+O(\epsilon))\frac{x}{2W}\rho \left( \frac{\log x}{\log y}\right),
\end{align*}
and this implies
\begin{align*}
\Sigma_3 &\leq (1+O(\epsilon))1024 \rho \left( \frac{\log x}{\log y}\right)\epsilon \cdot e^{-\gamma}\frac{W^2}{\varphi(W)^3}\frac{x}{2(\log x)^2\log((\log x)^{100})}.
\end{align*}
With $y = x^{3/10}$ we claim that $\rho(\frac{\log x}{\log y}) = \rho(10/3)\leq 0.025$. Since $1024 \cdot 0.025 = 25.6 < 26$ we obtain
\begin{align*}
\Sigma_3 &\leq 26\epsilon \cdot e^{-\gamma}\frac{W^2}{\varphi(W)^3}\frac{x}{2(\log x)^2\log((\log x)^{100})}.
\end{align*}

To see that $\rho(10/3) \leq 0.025$, one may use \eqref{eq:dickman rho differential delay} to show that
\begin{align*}
\rho(10/3) &= 1-\log(10/3) + \int_1^{7/3}\frac{\log u}{u+1} du - \int_1^{4/3}\frac{\log u}{u+1}(\log (10/3)-\log(u+2))du \\
&\leq 1-\log(10/3) + \int_1^{7/3}\frac{\log u}{u+1} du
\end{align*}
and then basic numerical integration establishes the inequality.
\end{proof}

\section{Exponential sums I--The proof of Proposition \ref{prop:sigma 1 is small}}\label{sec:sigma 1 is small}

We recall the definition of $\Sigma_1$ from \eqref{eq:defn of Sigma 1}:
\begin{align*}
\Sigma_1 &= \sum_{\substack{x/2 < p \leq x \\ p\equiv -1 (W) \\ \|\frac{\sigma_4(p+1)}{p(p+1)} + \frac{1}{16} \| \leq (\log x)^{-100}}} \mathbf{1}_{P^-(p+2) > x^{1/4}(\log x)^{100}}\cdot \mathbf{1}_{P^-(\frac{p+3}{2}) > (\log x)^{100}}.
\end{align*}

The first step in proving Proposition \ref{prop:sigma 1 is small} is to show that $p+1$ has a prime factor of ``convenient'' size. That is, it is rare for $p+1$ to have no prime factor between $x^{\epsilon^3}$ and $x^\epsilon$.

\begin{lemma}[Few shifted primes lack convenient factor]\label{lem:sigma 1 p plus 1 has a convenient factor}
For $\epsilon > 0$ sufficiently small and $x$ sufficiently large we have
\begin{align*}
\Sigma_1' := \sum_{\substack{x/2 < p \leq x \\ p\equiv -1 (W)\\ (p+1, P(x^{\epsilon^3},x^\epsilon))=1}} \mathbf{1}_{P^-(p+2) > x^{1/4}(\log x)^{100}}\cdot \mathbf{1}_{P^-(\frac{p+3}{2}) > (\log x)^{100}} \ll \epsilon^2\frac{W^2}{\varphi(W)^3} \frac{x}{(\log x)^2 \log\log x}.
\end{align*}
\end{lemma}

We also need an estimate for exponential sums of a certain shape.

\begin{lemma}[Basic exponential sum estimate]\label{lem:sigma1 exp sum}
Let $Q\geq 1$ be sufficiently large, and let $U\geq 5$ be fixed. Let $A, B$ be real numbers with $Q^5 \leq |A| \leq Q^U$ and $|B| \leq |A|$. Then
\begin{align*}
\left|\sum_{n\asymp Q} e \left(A(n^2+n^{-2}) + B(n+n^{-3}) \right) \right| \ll_{C,U} \frac{Q}{(\log Q)^C}
\end{align*}
for any fixed $C>0$.
\end{lemma}

We finish the proof of Proposition \ref{prop:sigma 1 is small} contingent upon Lemmas \ref{lem:sigma 1 p plus 1 has a convenient factor} and \ref{lem:sigma1 exp sum}, and then prove the two lemmas.

\begin{proof}[Proof of Proposition \ref{prop:sigma 1 is small} assuming Lemmas \ref{lem:sigma 1 p plus 1 has a convenient factor} and \ref{lem:sigma1 exp sum}]
Note that
\begin{align*}
\frac{\sigma_4(p+1)}{p(p+1)} = \frac{\sigma_4(p+1)}{(p+1)^2}+\frac{\sigma_4(p+1)}{(p+1)^3} + \frac{\sigma_4(p+1)}{(p+1)^4} + O(x^{-1}),
\end{align*}
so condition \eqref{eq:p+1 frac part with no small prime factors} becomes
\begin{align*}
\left\|\frac{\sigma_4(p+1)}{(p+1)^2}+\frac{\sigma_4(p+1)}{(p+1)^3} + \frac{\sigma_4(p+1)}{(p+1)^4} + \frac{1}{16} \right\| \leq (\log x)^{-99},
\end{align*}
say.

We split the sum over $p$ in $\Sigma_1$ according to whether or not $p+1$ has a prime divisor $q$ in $(x^{\epsilon^3},x^\epsilon]$. Lemma \ref{lem:sigma 1 p plus 1 has a convenient factor} shows that the contribution of those $p$ with $p+1$ having no prime factor in this range is negligible. For the sum over $p$ in which $p+1$ does have a convenient prime factor $q$, we no longer need to keep track of the condition that $p$ is prime, or that $p+2$ and $\frac{p+3}{2}$ have no small prime factors. If we change variables $mq=p+1$ and drop conditions by positivity we have
\begin{align}\label{eq:sigma 1 is error plus sigma 1 circ}
\Sigma_1 &\leq O\left(\epsilon^2\frac{W^2}{\varphi(W)^3} \frac{x}{(\log x)^2 \log \log x} \right) + \Sigma_1^\circ,
\end{align}
where
\begin{align*}
\Sigma_1^\circ &:= \mathop{\sum\sum}_{\substack{x/4 < m \leq 2x \\ x^{\epsilon^3} < q \leq x^\epsilon \\ q \mid m \\ \|\frac{\sigma_4(mq)}{(mq)^2} + \frac{\sigma_4(mq)}{(mq)^3}  + \frac{\sigma_4(mq)}{(mq)^4} + \frac{1}{16} \|\leq (\log x)^{-50}}} 1.
\end{align*}
By trivial estimation we may assume $q$ and $m$ are coprime, and then by multiplicativity we have $\sigma_4(mq)=\sigma_4(m)\sigma_4(q)$. An important point now is that
\begin{align*}
\frac{\sigma_4(q)}{q^2} = q^2 + q^{-2}, \ \ \ \ \ \ \frac{\sigma_4(q)}{q^3} = q + q^{-3}.
\end{align*}
After dropping the condition $(m,q)=1$ by positivity, we have
\begin{align*}
\Sigma_1^\circ &\leq \mathop{\sum\sum}_{\substack{x/4 < mq \leq 2x \\ x^{\epsilon^3} < q \leq x^\epsilon \\ \| \frac{\sigma_4(m)}{m^2}(q^2+q^{-2}) + \frac{\sigma_4(m)}{m^3}(q+q^{-3}) + \frac{\sigma_4(m)}{m^4} + \frac{1}{16}\| \leq (\log x)^{-10}}} 1.
\end{align*}
We put $q$ into dyadic intervals $q \asymp Q$ and then drop the condition that $q$ is prime, so that
\begin{align*}
\Sigma_1^\circ &\ll (\log x) \max_{x^{\epsilon^3} \ll Q \ll x^\epsilon} \mathop{\sum\sum}_{\substack{m\asymp x/Q \\ n \asymp Q \\ \| \frac{\sigma_4(m)}{m^2}(n^2+n^{-2}) + \frac{\sigma_4(m)}{m^3}(n+n^{-3}) + \frac{\sigma_4(m)}{m^4} + \frac{1}{16}\| \leq (\log x)^{-10}}} 1.
\end{align*}

We can, at last, handle the condition on $\| \cdot \|$. We insert a smooth, non-negative, 1-periodic function $w(y)$ with $\mathbf{1}_{\|y\| \leq (\log x)^{-10}} \leq w(y)$. We may choose $w$ so that $\widehat{w}(0) \ll (\log x)^{-10}$ and $w^{(j)}(y) \ll_j (\log x)^{10j}$. By repeated integration by parts we have the bound
\begin{align*}
\widehat{w}(h) = \int_0^1 w(y) e(-hy) dy \ll_j (\log x)^{-10} \left(1 + \frac{|h|}{(\log x)^{10}} \right)^{-j}
\end{align*}
for any integral $j\geq 0$, and by Fourier expansion we obtain
\begin{align*}
\Sigma_1^\circ \ll x(\log x)^{-100}+ &(\log x)\max_{x^{\epsilon^3} \ll Q \ll x^\epsilon} \sum_{|h| \leq (\log x)^{20}}|\widehat{w}(h)| \\
&\times\sum_{m\asymp x/Q}\left|\sum_{n\asymp Q} e \left(h \left(\frac{\sigma_4(m)}{m^2}(n^2+n^{-2}) + \frac{\sigma_4(m)}{m^3}(n+n^{-3}) \right)\right) \right|.
\end{align*}
The contribution of $h=0$ is sufficiently small for Proposition \ref{prop:sigma 1 is small}. For each $h$ with $1\leq |h| \leq (\log x)^{20}$ we apply Lemma \ref{lem:sigma1 exp sum} with $A = h\frac{\sigma_4(m)}{m^2}$, $B = h\frac{\sigma_4(m)}{m^3}$, and $C\geq 1$ sufficiently large. It follows that $\Sigma_1^\circ \ll x(\log x)^{-9}$ and by \eqref{eq:sigma 1 is error plus sigma 1 circ} we have
\begin{align*}
\Sigma_1 &\ll \epsilon^2\frac{W^2}{\varphi(W)^3} \frac{x}{(\log x)^2 \log \log x}. \qedhere
\end{align*}
\end{proof}

\begin{proof}[Proof of Lemma \ref{lem:sigma 1 p plus 1 has a convenient factor}]
We relax the condition $P^-(p+2)$ and insert upper-bound sieves via \eqref{eq:linear sieve inequalities} and \eqref{eq:fundamental lemma}. If $D = x^{1/10}$, and $R = \lfloor (\log \log x)^2 \rfloor$, then
\begin{align*}
\Sigma_1' &\leq \sum_{\substack{x/2 < p \leq x \\ p\equiv -1 (W)}}  \Big(\sum_{\substack{d\mid P(x^{1/100}) \\ d\leq D \\ d \mid p+2 \\ (d,W)=1}}\lambda_d^+ \Big)\Big(\sum_{\substack{e\mid P(x^{\epsilon^3},x^\epsilon) \\ e\leq D \\ e\mid p+1}}\lambda_e^+ \Big)\Big(\sum_{\substack{f\mid P((\log x)^{100}) \\ \omega(f) \leq 2R \\ f\mid p+3 \\ (f,W)=1}} \mu(f)\Big).
\end{align*}
We interchange the order of summation, and note that the variables $d,e$, and $f$ are pairwise coprime. We may use the Bombieri-Vinogradov to bound the error for counting primes in arithmetic progressions as in the proof of Lemma \ref{lem:lower bound on shifted prime} and obtain
\begin{align*}
\Sigma_1' &\leq \frac{\pi(x)-\pi(x/2)}{\varphi(W)}\sum_{\substack{d\mid P(x^{1/100}) \\ d\leq D \\ (d,W)=1}}\frac{\lambda_d^+}{\varphi(d)} \sum_{\substack{e\mid P(x^{\epsilon^3},x^\epsilon) \\ e\leq D \\ (e,dW)=1}}\frac{\lambda_e^+}{\varphi(e)}\sum_{\substack{f\mid P((\log x)^{100}) \\ \omega(f) \leq 2R \\ (f,deW)=1}}\frac{\mu(f)}{\varphi(f)} + O \left(\frac{x}{(\log x)^{100}} \right).
\end{align*}
By the usual arguments we then obtain
\begin{align*}
\Sigma_1' &\ll \frac{x}{\varphi(W)\log x} \prod_{p > D_0}(1+O(p^{-2}))\prod_{D_0 < p \leq x^{1/100}} \left(1 - \frac{1}{p}\right) \prod_{x^{\epsilon^3} < p \leq x^\epsilon}\left(1 - \frac{1}{p}\right)\prod_{D_0 < p \leq (\log x)^{100}} \left(1 - \frac{1}{p}\right) \\
&\ll \epsilon^2\frac{W^2}{\varphi(W)^3} \frac{x}{(\log x)^2 \log \log x},
\end{align*}
the last inequality following from several applications of Mertens' theorem.
\end{proof}

\begin{proof}[Proof of Lemma \ref{lem:sigma1 exp sum}]
We may assume without loss of generality that $Q$ is sufficiently large compared to $C$ and $U$. Let $E$ denote the sum over $n$ which we wish to bound. The first step is to apply a few steps of Weyl-van der Corput differencing to remove the positive powers of $n$ from the phase function. Write $K = \lfloor (\log Q)^{2C} \rfloor$. By \cite[Proposition 8.18]{IK2004} we deduce
\begin{align*}
|E|^2 &\ll \frac{Q^2}{K} + \frac{Q}{K}\sum_{1\leq |k| \leq K} \left|\sum_{n \asymp Q}e \left(A(2nk+(n+k)^{-2}-n^{-2}) + B((n+k)^{-3}-n^{-3}) \right) \right|.
\end{align*}
We then take the maximum over $k$ to obtain
\begin{align*}
|E|^2 &\ll \frac{Q^2}{K} + Q \left|\sum_{n \asymp Q}e \left(A(2nk+(n+k)^{-2}-n^{-2}) + B((n+k)^{-3}-n^{-3}) \right) \right|
\end{align*}
for some non-zero $k$ with $1\leq |k| \leq K$.

Now write $L = \lfloor (\log Q)^{4C}\rfloor$ and perform Weyl-van der Corput differencing again. We then have
\begin{align*}
|E|^4 &\ll \frac{Q^4}{K^2} + \frac{Q^4}{L} + Q^3\left|\sum_{n \asymp Q} e(f_\ell(n)) \right|
\end{align*}
for some $\ell$ with $1\leq |\ell| \leq L$, where
\begin{align*}
f_\ell(n) &= A \left((n+\ell+k)^{-2}-(n+k)^{-2} - (n+\ell)^{-2}+n^{-2} \right) \\
&+ B\left((n+\ell+k)^{-3}-(n+k)^{-3} - (n+\ell)^{-3}+n^{-3} \right).
\end{align*}
By the fundamental theorem of calculus we find
\begin{align*}
f_\ell(n) &= 6A \int_0^k\int_0^\ell \frac{1}{(n+s+t)^4} ds  dt + 12B\int_0^k\int_0^\ell \frac{1}{(n+s+t)^5} ds  dt,
\end{align*}
and since $|B| \leq A$ we see the derivatives of $f_\ell(n)$ satisfy
\begin{align*}
\Big|f_\ell^{(j)}(n) \Big| \asymp_J \frac{A|k\ell|}{Q^4} Q^{-j}, \ \ \ \ \ 0\leq j \leq J,
\end{align*}
where $J$ is any large, fixed constant. We claim there exists a positive constant $c=c(C,U)$ depending on $C$ and $U$ such that 
\begin{align*}
\left|\sum_{n \asymp Q} e(f_\ell(n)) \right| \ll_{C,U} Q^{1-c}
\end{align*}
since $Q^5 \leq |A| \leq Q^U$. Indeed, this follows from \cite[Theorem 8.4]{IK2004} with $F\asymp A|k\ell|Q^{-4}\gg Q$ and taking $k$ there sufficiently large depending on $C,U$ (see also \cite[Theorem 2.9]{GK1991}). It follows that $|E|^4 \ll_{C,U} \frac{Q^4}{K^2} + \frac{Q^4}{L} + Q^{4-c}$ and this suffices for the proof.
\end{proof}

\section{Exponential sums II--the proof of Proposition \ref{prop:sigma 4 is small}}\label{sec:sigma 4 is small}

From \eqref{eq:defn of Sigma 4} we have
\begin{align*}
\Sigma_4 &= \sum_{x^{1/4-\epsilon} < r \leq x^{1/4}(\log x)^{100}}\sum_{\substack{x/2 < p \leq x \\ p\equiv -1 (W) \\ p+2 \equiv 0 (r) \\ p+1 \text{ not } x^{3/10}\text{-smooth} \\ \|\frac{\sigma_4(p+1)}{p(p+1)} + \frac{p+1}{r^4} + \frac{1}{16} \| \leq (\log x)^{-100}}} \mathbf{1}_{P^-(p+2) > x^{1/4-\epsilon}}\cdot \mathbf{1}_{P^-(\frac{p+3}{2}) > (\log x)^{100}},
\end{align*}
and to prove Proposition \ref{prop:sigma 4 is small} we must show
\begin{align*}
\Sigma_4 &\ll \epsilon^2\frac{W^2}{\varphi(W)^3} \frac{x}{(\log x)^2 \log\log x}.
\end{align*}

We need the following lemmas on exponential sums.

\begin{lemma}[Exponential sum estimate for small $Q$]\label{lem:sigma 4 exp sum, Q less than 5 12}
Let $\epsilon > 0$ be sufficiently small, and let $x$ be sufficiently large depending on $\epsilon$. Let $r$ be an integer satisfying $x^{1/4-\epsilon}\leq r \leq x^{1/4+\epsilon}$, and let $Q$ satisfy $x^{3/10}\ll Q \ll x^{5/12-100\epsilon}$. Let $m \asymp x/Q$ and $1\leq |h| \leq (\log x)^{100}$ be integers. Then for any integer $v = v(r)$ with $|v| \leq r$ we have
\begin{align*}
\left|\sum_{\substack{\ell \asymp Q/r}} e\left(h \left(\frac{\sigma_4(m)}{m^2}(2vr\ell + r^2 \ell^2+(v+r\ell)^{-2}) + \frac{\sigma_4(m)}{m^3}r\ell + \frac{m\ell}{r^3} \right) \right) \right| \ll_{C,\epsilon} \frac{Q}{r (\log x)^C}.
\end{align*}
\end{lemma}

\begin{lemma}[Exponential sum estimate for large $Q$]\label{lem:sigma exp sum, Q bigger than 5 12}
Let $\epsilon > 0$ be sufficiently small, and let $x$ be sufficiently large depending on $\epsilon$. Let $r$ be an integer satisfying $x^{1/4-\epsilon}\leq r \leq x^{1/4+\epsilon}$, and let $Q$ satisfy $x^{5/12+100\epsilon}\ll Q \ll x^{1-\epsilon}$. Let $\alpha(n)$ denote the indicator function of $n$ having a prime divisor in the interval $(x^{\epsilon^3},x^{\epsilon^2}]$, and let $1\leq |h| \leq (\log x)^{100}$ be an integer. Then for any integers $v_m$ with $|v_m| \leq r$ we have
\begin{align*}
\sum_{\substack{m\asymp x/Q \\ \alpha(m)=1}} \left|\sum_{\substack{\ell \asymp Q/r}} e\left(h \left(\frac{\sigma_4(m)}{m^2}(2v_mr\ell + r^2 \ell^2+(v_m+r\ell)^{-2}) + \frac{\sigma_4(m)}{m^3}r\ell + \frac{m\ell}{r^3} \right) \right) \right| \ll_{C,\epsilon} \frac{x}{r (\log x)^C}.
\end{align*}
\end{lemma}

\begin{proof}[Proof of Proposition \ref{prop:sigma 4 is small} assuming Lemmas \ref{lem:sigma 4 exp sum, Q less than 5 12} and \ref{lem:sigma exp sum, Q bigger than 5 12}]
Let $\alpha(p+1)$ be the indicator function of $p+1$ having a prime divisor in $(x^{\epsilon^3},x^{\epsilon^2}]$. By slightly adapting the proof of Lemma \ref{lem:sigma 1 p plus 1 has a convenient factor} one shows that the terms in $\Sigma_4$ with $\alpha(p+1)=0$ have a total contribution of
\begin{align*}
&\ll \epsilon^2\frac{W^2}{\varphi(W)^3} \frac{x}{(\log x)^2 \log\log x}.
\end{align*}
Therefore, we may assume we sum over $p$ in $\Sigma_4$ such that $\alpha(p+1)=1$.

We change variables $p+1= mq$, where $q$ is a prime $\gg x^{3/10}$; the condition $\alpha(p+1)=1$ is equivalent to $\alpha(m)=1$. We then break the sum over $q$ into dyadic segments $q \asymp Q$. Therefore, up to errors of size $\ll\epsilon^2\frac{W^2}{\varphi(W)^3} \frac{x}{(\log x)^2 \log\log x}$ we have
\begin{align*}
\Sigma_4 &\leq \sum_{x^{3/10}\ll Q=2^j\ll x/W}\sum_{x^{1/4-\epsilon} < r \leq x^{1/4}(\log x)^{100}} \sum_{\substack{m\asymp x/Q \\ m\equiv 0 (W) \\ (m,r)=1 \\ \alpha(m)=1}} \sum_{\substack{q \asymp Q\\ mq+1 \equiv 0 (r) \\ P^-(mq-1) > x^{1/100} \\ P^-(mq+1) > x^{1/100} \\ P^-(\frac{mq+2}{2}) > (\log x)^{100} \\ \|\frac{\sigma_4(mq)}{mq (mq-1)} + \frac{mq}{r^4} + \frac{1}{16} \| \leq (\log x)^{-100}}} 1.
\end{align*}

The last step before proceeding to exponential sums is to eliminate some inconvenient ranges of $Q$. For fixed $Q,r$, and $m$, we use upper-bound sieves to see that the sum over primes $q$ is
\begin{align*}
&\leq \sum_{\substack{n \asymp Q \\ n\equiv -\overline{m} (r)}}\Big(\sum_{\substack{d_1 \mid P((\log x)^{100}) \\ \omega(d_1)\leq 2R \\ d_1 \mid n}} \mu(d_1) \Big)\Big(\sum_{\substack{d_2 \mid P((\log x)^{100}, x^{1/100}) \\ d_2 \leq D \\ d_2 \mid n}}\lambda_{d_2}^+ \Big) \\
&\times\Big(\sum_{\substack{e_1 \mid P((\log x)^{100}) \\ \omega(e_1)\leq 2R \\ e_1 \mid mn-1 \\ (e_1,W)=1}} \mu(e_1) \Big)\Big(\sum_{\substack{e_2 \mid P((\log x)^{100}, x^{1/100}) \\ e_2 \leq D \\ e_2 \mid mn-1}}\lambda_{e_2}^+ \Big)\Big( \sum_{\substack{f \mid P(x^{1/100}) \\ f \leq D \\ f \mid mn+1 \\(f,W)=1}}\lambda_f^+\Big)\Big(\sum_{\substack{g \mid P((\log x)^{100}) \\ \omega(g)\leq 2R \\ g \mid mn+2 \\ (g,W)=1}} \mu(g) \Big),
\end{align*}
where $D = x^{1/80}$ and $R = \lfloor (\log \log x)^2 \rfloor$. This is similar to, and only slightly more complicated than, sieve sums we have already encountered. We therefore skip the details and just record that this sum over $n$ is
\begin{align*}
&\ll \frac{W^3}{\varphi(W)^3} \frac{Q}{r (\log x)^3 \log \log x}.
\end{align*}
It follows the contribution to $\Sigma_4$ from $Q$ which satisfy $x^{5/12-100\epsilon} \ll Q \ll x^{5/12+100\epsilon}$ is
\begin{align*}
\ll\epsilon^2\frac{W^2}{\varphi(W)^3} \frac{x}{(\log x)^2 \log\log x}.
\end{align*}
The contribution from $Q \gg x^{1-\epsilon}$ is similarly negligible, so we may assume $Q \ll x^{1-\epsilon}$.

For the remaining ranges of $Q$ we may impose the condition that $(m,q)=1$ by trivial estimation, then use multiplicativity to write $\sigma_4(mq) = \sigma_4(m)(q^4+1)$, and then drop the coprimality condition and the condition that $q$ is prime by positivity. Therefore
\begin{align*}
\Sigma_4 \ll \epsilon^2\frac{W^2}{\varphi(W)^3} \frac{x}{(\log x)^2 \log\log x} &+ (\log x) \max_{\substack{x^{3/10} \ll Q \ll x^{1-\epsilon} \\ Q \not \in [x^{5/12-100\epsilon},x^{5/12+100\epsilon}]}}\sum_{x^{1/4-\epsilon} < r \leq x^{1/4}(\log x)^{100}} \sum_{\substack{m\asymp x/Q \\ \alpha(m)=1}} \\ 
&\times\sum_{\substack{n \asymp Q\\ n\equiv -\overline{m} (r) \\ \|\frac{\sigma_4(m)}{m^2}(n^2+n^{-2}) + \frac{\sigma_4(m)}{m^3}n + \frac{\sigma_4(m)}{m^4}+ \frac{mn}{r^4} + \frac{1}{16} \| \leq (\log x)^{-50}}} 1.
\end{align*}
We introduce a smooth function $w$ as in the proof of Proposition \ref{prop:sigma 1 is small}, and then expand $w$ in its Fourier series. The contribution of the zero frequency is $\ll x(\log x)^{-40}$, the contribution of the frequencies $|h| > (\log x)^{60}$ is negligible, and the contribution of the intermediate frequencies is
\begin{align*}
\ll (\log x)^{O(1)}\sum_{x^{1/4-\epsilon} < r \leq x^{1/4}(\log x)^{100}} \sum_{\substack{m\asymp x/Q \\ \alpha(m)=1}}\Big|\sum_{\substack{n \asymp Q \\ n\equiv -\overline{m} (r)}} e\left(h \left(\frac{\sigma_4(m)}{m^2}(n^2+n^{-2}) + \frac{\sigma_4(m)}{m^3}n + \frac{mn}{r^4} \right) \right) \Big|
\end{align*}
for some $x^{3/10}\ll Q \ll x^{1-\epsilon}$ which avoids $x^{5/12}$ and some integral $h$ satisfying $1\leq |h| \leq (\log x)^{60}$.

We handle the congruence condition on $n$ by changing variables $n = v_m+r\ell$, where $v_m = v(m,r)$ is an integer satisfying $|v_m| \leq r$. Up to an error of size $O(x^{1-\epsilon})$ the contribution of the intermediate frequencies is
\begin{align*}
&\ll (\log x)^{O(1)}\sum_{x^{1/4-\epsilon} < r \leq x^{1/4}(\log x)^{100}} \sum_{\substack{m\asymp x/Q \\ \alpha(m)=1}} \\
&\times \left|\sum_{\substack{\ell \asymp Q/r}} e\left(h \left(\frac{\sigma_4(m)}{m^2}(2v_mr\ell + r^2 \ell^2+(v_m+r\ell)^{-2}) + \frac{\sigma_4(m)}{m^3}r\ell + \frac{m\ell}{r^3} \right) \right) \right|.
\end{align*}

If $Q \ll x^{5/12-100\epsilon}$ then Lemma \ref{lem:sigma 4 exp sum, Q less than 5 12} shows the sum over $\ell$ has size $\ll_A \frac{Q}{r}(\log x)^{-A}$ for any large fixed $A$, and this is acceptably small. We may therefore assume $Q \gg x^{5/12+100\epsilon}$. In this case, Lemma \ref{lem:sigma exp sum, Q bigger than 5 12} shows that the double sum over $m$ and $\ell$ has size $\ll_A \frac{x}{r}(\log x)^{-A}$ for any large, fixed $A$. It follows that
\begin{align*}
\Sigma_4 &\ll_A \epsilon^2\frac{W^2}{\varphi(W)^3} \frac{x}{(\log x)^2 \log\log x} + \frac{x}{(\log x)^A}
\end{align*}
and this is sufficient for Proposition \ref{prop:sigma 4 is small}.
\end{proof}

\begin{proof}[Proof of Lemma \ref{lem:sigma 4 exp sum, Q less than 5 12}]
The proof is very similar to the proof of Lemma \ref{lem:sigma1 exp sum}. If we let $E$ denote the sum over $\ell$ we wish to bound then two rounds of Weyl-van der Corput differencing as in the proof of Lemma \ref{lem:sigma1 exp sum} give
\begin{align*}
|E|^4 \ll \frac{Q^4}{r^4 J^2} + \frac{Q^4}{r^4 K} &+ \frac{Q^3}{r^3}\Big|\sum_{\ell \asymp Q/r} e\Big(h \frac{\sigma_4(m)}{m^2}((v+r(\ell+j+k))^{-2} \\
&-(v+r(\ell+k))^{-2}- (v+r(\ell+j))^{-2}+(v+r\ell)^{-2}) \Big)\Big|
\end{align*}
for integers $j$ and $k$ satisfying $1\leq |j| \leq J =\lfloor (\log x)^{C_1}\rfloor$, $1\leq |k| \leq K = \lfloor (\log x)^{C_2}\rfloor$. Here $C_1$ and $C_2$ are sufficiently large, fixed integers. If we let $\Phi(\ell)$ denote the phase function inside the exponential then
\begin{align*}
\Phi(\ell) = 6hr^2\frac{\sigma_4(m)}{m^2}\int_0^j \int_0^k \frac{1}{(v+r(\ell+s+t))^4} ds \ dt
\end{align*}
and therefore
\begin{align*}
\left|\Phi^{(n)}(\ell) \right| &\asymp_n |hjk| \frac{x^2}{Q^2} r^2 \frac{1}{r^4 \ell^4} \ell^{-n} \asymp_n |hjk| \frac{x^2r^2}{Q^6} (Q/r)^{-n}.
\end{align*}
Since $Q \ll x^{5/12-100\epsilon}$ we see that $|hjk| \frac{x^2r^2}{Q^6} \gg x^{10\epsilon}$, say, and since $Q \gg x^{3/10}$ we have by \cite[Theorem 8.4]{IK2004} that
\begin{align*}
\left|\sum_{\ell \asymp Q/r} e(\Phi(\ell)) \right| \ll_\epsilon (Q/r)^{1-c},
\end{align*}
where $c$ is a positive constant which depends on $\epsilon, C_1$, and $C_2$. This completes the proof.
\end{proof}

\begin{proof}[Proof of Lemma \ref{lem:sigma exp sum, Q bigger than 5 12}]
Let $\Upsilon$ denote the sum over $m$ and $\ell$ which we wish to bound.

We need to take advantage of the averaging over $m$, but the fact that the integer $v_m$ depends on $m$ at first prevents this. Our first task, then, is to remove $v_m$.

By series expansion we have
\begin{align*}
\frac{1}{(v_m+r\ell)^2} = \frac{1}{r^2\ell^2}-\frac{2v_m}{r^3\ell^3} + O \left(\frac{v_m^2}{r^4 \ell^4} \right),
\end{align*}
and
\begin{align*}
\frac{\sigma_4(m)}{m^2} \frac{v_m^2}{r^4 \ell^4} \ll \frac{x^2}{Q^2} \frac{1}{r^2 \ell^4} \ll \frac{x^2r^2}{Q^6}.
\end{align*}
Since $r\ll x^{1/4+\epsilon}$ and $Q \gg x^{5/12+100\epsilon}$ we see this is $O(x^{-10\epsilon})$, say, so up to an error of size $\ll \frac{x^{1-\epsilon}}{r}$ we have
\begin{align*}
|\Upsilon| &= \sum_{\substack{m\asymp x/Q \\ \alpha(m)=1}} \Big|\sum_{\substack{\ell \asymp Q/r}} e\left(h \left(\frac{\sigma_4(m)}{m^2}(2v_mr\ell + r^2 \ell^2+\frac{1}{r^2\ell^2} - \frac{2v_m}{r^3\ell^3}) + \frac{\sigma_4(m)}{m^3}r\ell + \frac{m\ell}{r^3} \right) \right) \Big|.
\end{align*}
Set $J = \lfloor (\log x)^{B}\rfloor$ with $B$ sufficiently large, then shift the sum over $\ell$ by $j$ and average over $1\leq j \leq J$. This gives
\begin{align*}
|\Upsilon| &\ll \frac{xJ}{Q} + \frac{1}{J}\sum_{\substack{m\asymp x/Q \\ \alpha(m)=1}}\sum_{\ell \asymp Q/r} \\ 
&\Big|\sum_{1\leq j \leq J}e\left(h \left(\frac{\sigma_4(m)}{m^2}(2v_mrj + r^2 (\ell+j)^2+\frac{1}{r^2(\ell+j)^2} - \frac{2v_m}{r^3(\ell+j)^3}) + \frac{\sigma_4(m)}{m^3}rj + \frac{mj}{r^3} \right) \right) \Big|.
\end{align*}
We then apply Cauchy-Schwarz, interchange the order of summation, and change variables to obtain
\begin{align*}
|\Upsilon|^2 &\ll \frac{x^2}{r^2 J} + \frac{x}{rJ}\sum_{1\leq |j| \leq J} \sum_{\substack{m\asymp x/Q \\ \alpha(m)=1}}\Big| \sum_{\ell \asymp Q/r} e\left(h\frac{\sigma_4(m)}{m^2}\left(2jr^2\ell + \frac{1}{r^2(\ell+j)^2} - \frac{1}{r^2\ell^2} +O\left( \frac{|jv_m|}{r^3\ell^4}\right)\right) \right)\Big|.
\end{align*}
Since $Q \gg x^{5/12+100\epsilon}$ and $|v_m|\leq r$ we have
\begin{align*}
h\frac{\sigma_4(m)}{m^2} \cdot \frac{j|v_m|}{r^3\ell^4} \ll (\log x)^{O(1)} \frac{x^2r^2}{Q^6} \ll x^{-10\epsilon},
\end{align*}
and therefore
\begin{align*}
|\Upsilon|^2 &\ll \frac{x^2}{r^2 J} + \frac{x}{r}\sum_{\substack{m\asymp x/Q \\ \alpha(m)=1}}\Big| \sum_{\ell \asymp Q/r} e\left(h\frac{\sigma_4(m)}{m^2}\left(2jr^2\ell + \frac{1}{r^2(\ell+j)^2} - \frac{1}{r^2\ell^2}\right) \right)\Big|
\end{align*}
for some $1\leq |j| \leq J$. We have therefore eliminated $v_m$ from the sum over $\ell$.

We now use the condition $\alpha(m)=1$, which we recall means $m$ has a prime divisor in the interval $(x^{\epsilon^3},x^{\epsilon^2}]$. We change variables $m \rightarrow m t$, where $t$ is the prime divisor in question. After dyadically decomposing the range of $t$ we obtain
\begin{align*}
|\Upsilon|^2 & \ll \frac{x^2}{r^2 J} + \frac{x(\log x)}{r}\max_{\substack{x^{\epsilon^3} \ll T \ll x^{\epsilon^2}}}\mathop{\sum\sum}_{\substack{m\asymp x/QT \\ t \asymp T}}\left| \sum_{\ell \asymp Q/r} e\left(h\frac{\sigma_4(mt)}{(mt)^2}\left(2jr^2\ell + \frac{1}{r^2(\ell+j)^2} - \frac{1}{r^2\ell^2}\right) \right)\right|.
\end{align*}
The contribution from $t \mid m$ is negligible, so we may assume $(t,m)=1$ and therefore $\frac{\sigma_4(mt)}{(mt)^2} = \frac{\sigma_4(m)}{m^2}(t^2+t^{-2})$. With $m$ and $t$ separated we may drop the condition that they are coprime, and we may also drop the condition that $t$ is prime. This yields
\begin{align*}
|\Upsilon|^2  \ll \frac{x^2}{r^2 J} &+ \max_{\substack{x^{\epsilon^3} \ll T \ll x^{\epsilon^2} \\ m\asymp x/QT}}\frac{x^2(\log x)}{QTr}\mathop{\sum}_{\substack{n \asymp T}} \\
&\times\Big| \sum_{\ell \asymp Q/r} e\left(h\frac{\sigma_4(m)}{m^2}(n^2+n^{-2})\left(2jr^2\ell + \frac{1}{r^2(\ell+j)^2} - \frac{1}{r^2\ell^2}\right) \right)\Big|.
\end{align*}
By the Cauchy-Schwarz inequality
\begin{align*}
\Big(\sum_{n\asymp T}\Big|\sum_{\ell \asymp Q/r} \Big|\Big)^2 &\ll \frac{QT^2}{r} + T\mathop{\sum\sum}_{\substack{\ell_1, \ell_2 \asymp Q/r \\ \ell_1 \neq \ell_2}}\Big|\sum_{n\asymp T} e \left(A_{\ell_1,\ell_2}(n^2+n^{-2}) \right) \Big|,
\end{align*}
where $A_{\ell_1,\ell_2}$ is a real number satisfying $|A_{\ell_1,\ell_2}| \asymp |hj| \frac{x^2 r^2}{Q^2 T^2} |\ell_1-\ell_2|$. Since $x^{\epsilon^3} \ll T \ll x^{\epsilon^2}$ and $x^{1/3} \ll |A_{\ell_1,\ell_2}| \ll x^2$ for $\ell_1 \neq \ell_2$, Lemma \ref{lem:sigma1 exp sum} implies
\begin{align*}
\left|\sum_{n\asymp T} e \left(A_{\ell_1,\ell_2}(n^2+n^{-2}) \right) \right| \ll_\epsilon T^{1-c}
\end{align*}
for some positive constant $c$. This completes the proof.
\end{proof}

\bibliographystyle{amsalpha}

\end{document}